\documentclass[reqno]{amsart}
\usepackage{hyperref}
\usepackage{graphicx}
\usepackage{amssymb}
\usepackage{enumerate}
\usepackage{xcolor}

\begin{document}

\title[Gromov-Hausdorff stability of inertial manifolds]
{Gromov-Hausdorff stability of \\ inertial manifolds
under perturbations \\ of the domain and equation}

\author[J. Lee and N. Nguyen]{Jihoon Lee and Ngocthach Nguyen$^{*}$}

\address{Jihoon Lee\hfill\break
Department of Mathematics, Chungnam National University,\hfill\break
Daejeon 34134, Korea}
\email{jjhlee24@gmail.com}

\address{Ngocthach Nguyen\hfill\break
Department of Mathematics, Chungnam National University,\hfill\break
Daejeon 34134, Korea}
\email{ngngocthach91@gmail.com}

\thanks{$^{*}$ Corresponding author}

\begin{abstract} 
In this paper, we study the Gromov-Hausdorff stability and continuous dependence of the inertial manifolds under perturbations of the domain and equation. More precisely, we use the Gromov-Hausdorff distances between two inertial manifolds and two dynamical systems to consider the continuous dependence of the inertial manifolds and the stability of the dynamical systems on inertial manifolds induced by the reaction diffusion equations under perturbations of the domain and equation.
\end{abstract}


\maketitle

\numberwithin{equation}{section}
\newtheorem{theorem}{Theorem}[section]
\newtheorem{remark}{Remark}[section]
\newtheorem{definition}{Definition}[section]
\newtheorem{lemma}[theorem]{Lemma}
\newtheorem{corollary}[theorem]{Corollary}
\newtheorem{proposition}[theorem]{Proposition}
\newtheorem{example}[theorem]{Example}

\renewcommand{\footskip}{.5in}

\section{Introduction}

Let $\Omega_0$ be an open bounded domain in $\mathbb R^N$ with smooth boundary. We consider the following reaction diffusion equation
\begin{equation} \label{eq:f0}
\begin{cases}
\partial_t u-\Delta u= f_0(u)    &\mbox{ in }\Omega_0 \times (0,\infty),\\
u=0                   &\mbox{ on }\partial\Omega_0 \times (0,\infty),
\end{cases}
\end{equation}
where $f_0: \mathbb{R}\to \mathbb{R}$ is a $C^1$ function such that $f_0$ and $f_0'$ are bounded, and $f_0$ satisfies the dissipative condition, i.e.,
\begin{align*}
\limsup_{|s| \rightarrow \infty} \dfrac{f_0(s)}{s} <0.
\end{align*}
It is well known in \cite{ACR} that the problem \eqref{eq:f0} is well-posed in various function spaces.
Let $F_0: L^2(\Omega_0) \rightarrow L^2(\Omega_0)$ be the Nemytskii operator of $f_0$. It is clear that $F_0$ is Lipschitz since $f'_0$ is bounded, and we may assume Lip$F_0 > 1$.

Let ${\rm Diff}(\Omega_0)$ be the space of diffeomorphisms $h$ from $\Omega_0$ onto its image $\Omega_h:=h(\Omega_0) \subset \mathbb R^N$ with the $C^1$ topology. Let $\mathcal{F}$ be the collection of $C^1$ functions $f_h : \mathbb{R} \rightarrow \mathbb{R}$ $(h \in{\rm Diff(\Omega_0)})$ with the dissipative condition such that $\overline{d}_{C^1} (f_h,f_0) \leq d_{C^1}(h,id)$, where the metric $\overline{d}_{C^1}$ on $\mathcal{F}$ is given by 
$$
\overline{d}_{C^1} (f_h, f_{\tilde{h}}) := \min\{d_{C^1} (f_h, f_{\tilde{h}}),1\}~\mbox{for}~h, \tilde h \in {\rm Diff}(\Omega_0),
$$
where $id$ denotes the identity map on $\Omega_0$.
For each $h \in {\rm Diff} (\Omega_0)$, we consider a perturbation of equation \eqref{eq:f0}
\begin{equation} \label{eq:fh}
\begin{cases}
\partial_t u-\Delta u= f_h(u)    &\mbox{ in }\Omega_h \times (0,\infty),\\
u=0                   &\mbox{ on }\partial\Omega_h \times (0,\infty).
\end{cases}
\end{equation}
As the above, we know that the problem \eqref{eq:fh} is well-posed, and the Nemytskii operator $F_h: L^2(\Omega_h) \rightarrow L^2(\Omega_h)$ of $f_h$ is Lipschitz.

For any $h \in {\rm Diff} (\Omega_0)$ $C^1$-close to $id$, we consider the following equation
\begin{equation}\label{eq:F}
u_t + A_h u = F_h(u),~u \in L^2 (\Omega_h),
\end{equation}
where $A_h$ denotes the operator $-\Delta$ on $\Omega_h$ with Dirichlet boundary condition.
For simplicity, we write $A_{id} =A_0$ and $F_{id} = F_{0}$. We know that $A_h$ has a family of eigenvalues $\{\lambda_i^h\}_{i=1}^{\infty}$ such that 
$$
0 < \lambda_1^h \le  \lambda_2^h \le \cdots \to \infty,
$$
and a family of corresponding eigenfunctions $\{\phi_i^h\}_{i = 1}^{\infty}$ which is an  orthonormal basis in $L^2(\Omega_h)$ and orthogonal in $H_0^1(\Omega_h)$. We denote the semi-dynamical system $S_h(t)$ induced by equation \eqref{eq:F} by
$$
S_h(t) : L^2(\Omega_h) \to L^2(\Omega_h),~S_h(t)(u_0) = u_h(t),~\forall t \ge 0,
$$
where $u_h(t)$ is the unique solution of \eqref{eq:F} with $u_h(0) = u_0$.

For any $h \in {\rm Diff}(\Omega_0)$ and $m\in \mathbb N$, let $P_m^h$ be the projection of $L^2(\Omega_h)$ onto ${\rm span}\{\phi_1^h, \ldots, \phi_m^h\}$, and $Q_m^h$ be the orthogonal complement of $P_m^h$. For simplicity, we will write $P_m^{id} := P_m^0$ and $Q_m^{id} := Q_m^0$.

\begin{definition}
We say that $\mathcal M \subset L^2(\Omega_0)$ is an $m$-dimensional inertial manifold of the semi-dynamical system $S(t)$ induced by \eqref{eq:f0} if it is the graph of a Lipschitz map $\Phi : P_m^0L^2(\Omega_0) \to Q_m^0L^2(\Omega_0)$ such that
\begin{itemize}
\item[(i)] $\mathcal M$ is invariant, i.e., $S(t)\mathcal M = \mathcal M$ for $t \in \mathbb R$,
\item[(ii)] $\mathcal M$ attracts all  trajectories of $S(t)$ exponentially, i.e., there are $C>0$ and $k>0$ such that for any $u_0 \in L^2(\Omega_0)$, there is $v_0 \in \mathcal M$ satisfying
$$
\|S(t)u_0 - S(t)v_0 \|_{L^2(\Omega_0)} \le Ce^{-kt}\|u_0 - v_0\|_{L^2(\Omega_0)},~\forall t>0.
$$
\end{itemize}
\end{definition}

In this paper, we are interested in studying the behavior of the inertial manifolds (which belong to disjoint phase spaces) of equation \eqref{eq:F} with respect to perturbations of the domain $\Omega_0$. More precisely, we use the Gromov-Hausdorff distances between two inertial manifolds  and two dynamical systems to consider the continuous dependence of the inertial manifolds and the stability of the dynamical systems on inertial manifolds induced by the reaction diffusion equations under perturbations of the domain and equation.
For this, we first prove the existence of inertial manifold of equation \eqref{eq:F} when $h$ is $C^1$-close enough to $id$.

\begin{theorem} \label{existence of IM}
Let the above assumptions on the operator $A_h$ and the nonlinearity $F_h$ hold and, in addition, let the following spectral gap condition hold:
\begin{equation}\label{spectral}
\lambda_{m+1}^0 - \lambda_{m}^0 >  2 \sqrt{2} L_0 ~\mbox{for~some}~m \in \mathbb N,
\end{equation}
where $L_0$ is a Lipschitz constant of the nonlinearity $F_0$, $\lambda_n^0$ is the $n$th eigenvalue of $A_0$ for $n \in \mathbb N$. Then there exists $\delta >0$ such that if $d_{C^1} (h,id)< \delta$, then equation \eqref{eq:F} admits the $m$-dimensional inertial manifold $\mathcal{M}_h$.
\end{theorem}

\begin{remark} \label{inertial_hist}
To the best of our knowledge, the existence of an inertial manifolds for \eqref{eq:f0} was first proved by Foias {et al.} \cite{FST, MS} with the non-optimal constant $C$ in the right hand side of assumption \eqref{spectral}. Moreover, Romanov \cite{Ro} proved the existence of inertial manifold of \eqref{eq:f0} under the spectral gap condition \eqref{spectral} using the Lyapunov-Perron method in \cite{SY}. For a detailed exposition of the classical theory of inertial manifolds, refer the paper by Zelik \cite{Z}.
\end{remark}

To study how the asymptotic dynamics of evolutionary equation \eqref{eq:F} changes when we vary the domain $\Omega_h$, our first task is to find a way to compare the inertial manifolds of the equations in different domains. One of the difficulties in this direction is that the phase space $L^2(\Omega_0)$ of the induced semi-dynamical system changes as we change the domain $\Omega_0$. In fact, the phase spaces $L^2(\Omega_0)$ and $L^2(\Omega_h)$ which contain inertial manifolds $\mathcal M_0$ and $\mathcal M_h$, respectively, can be disjoint even if $\Omega_h$ is a small perturbation of $\Omega_0$.

In this direction, Arrieta and Santamaria \cite{AS} estimated the distance of inertial manifolds $\mathcal{M}_\varepsilon$ of the following evolution problem
\begin{equation} \label{Fepsilon}
u_t + A_\varepsilon u = F_\varepsilon (u),~ \forall \varepsilon \in [0,\varepsilon_0]
\end{equation}
on the Hilbert spaces $X_\varepsilon$. For this purpose, they first assumed that the operator $A_0$ has the following spectral gap condition
$$
\lambda_{m+1}^0 - \lambda_{m}^0 \ge 18L_0 ~\mbox{and}~ \lambda_m^0 \ge 18L_0~\mbox{for~some}~ m \in \mathbb N
$$
to use the Lyapynov-Perron method for the existence of inertial manifold (see Proposition 2.1 in \cite{AS}). They also assumed that the nonlinear terms $F_{\varepsilon}$ have a uniformly bounded support, i.e., there exists $R>0$ such that
$$
{\rm supp} F_{\varepsilon} \subset D_R=\{u \in X_{\varepsilon}: \|u\|_{X_{\varepsilon}} \le R\},~ \forall \varepsilon \in [0,\varepsilon_0].
$$
This assumption implies that every inertial manifold $\mathcal M_{\varepsilon}$ of \eqref{Fepsilon} does not perturb outside the ball $D_R$ even though $\varepsilon$ varies. In fact, we have
$$
\mathcal M_{\varepsilon} \cap (X_{\varepsilon}\setminus D_R) = P_m^{\varepsilon}(X_{\varepsilon})\cap (X_{\varepsilon}\setminus D_R), ~\forall \varepsilon \in [0, \varepsilon_0].
$$
Note that the inertial manifold $\mathcal M_{\varepsilon}$ (or $\mathcal M_0$) of \eqref{Fepsilon} is expressed by the graph of a Lipschitz map $\Phi_{\varepsilon}$ (or $\Phi_0$).  Under the above assumptions, they proved 
$$
\|\Phi_\varepsilon - E_\varepsilon \Phi_0 \|_{L^\infty (\mathbb{R}^m, X_\varepsilon)} \rightarrow 0 \hbox{ as } \varepsilon \rightarrow 0,
$$ 
where $E_\varepsilon$ is an isomorphism from $X_0$ to $X_\varepsilon$ (for more details, see Theorem 2.3 in \cite{AS}).
Note that the norms $\| \cdot\|_{L^\infty (\mathbb{R}^m, X_\varepsilon)}$ and $\| \cdot\|_{L^\infty (\mathbb{R}^m, X_{\varepsilon'})}$ can not be comparable in general if $\varepsilon \neq \varepsilon'$. 
For any $\varepsilon \in [0, \varepsilon_0]$, we take $h_{\varepsilon} \in {\rm Diff}(\Omega_0)$ satisfying $d_{C^1}(h_{\varepsilon}, id) = \varepsilon$. Then the perturbed phase space $X_{\varepsilon}$ in \cite{AS} can be considered as the space $L^2(\Omega_{h_{\varepsilon}})$.

In this paper, we do not assume that the nonlinear terms $F_h$ $(h \in {\rm Diff}(\Omega_0))$  have a uniformly bounded support.

Recently, Lee \textit{et al.} \cite{L, LNT} introduced the Gromov-Hausdorff distance between two dynamical systems on compact metric spaces to analyze how
the asymptotic dynamics of the global attractors of \eqref{eq:f0} changes when we vary the domain $\Omega_0$.

To compare the asymptotic behavior of the dynamics on inertial manifolds, we first need to introduce the notion of Gromov-Hausdorff distance between two dynamical systems on noncompact metric spaces.
Let $(X,d_X)$ and $(Y,d_Y)$ be two metric spaces.
For any $\varepsilon >0$ and a subset $B$ of $X$, we recall that a map $i: X \rightarrow Y$ is an $\varepsilon$-isometry on $B$ if $|d_Y(i(x),i(y)) - d_X(x,y)| < \varepsilon$ for all $x,y \in B$.
In the case $B=X$, we say that $i: X \rightarrow Y$ is an $\varepsilon$-isometry.
The Gromov-Hausdorff distance $d_{GH}(X,Y)$ between $X$ and $Y$ is defined by the infimum of $\varepsilon >0$ such that there are $\varepsilon$-isometries $i: X \rightarrow Y$ and $j: Y \rightarrow X$ such that $U_\varepsilon (i(X)) =Y$ and $U_\varepsilon (j(Y))=X$, where $U_\varepsilon (B)$ is the $\varepsilon$ neighborhood of $B$.
Let $\mathcal{X} =\{X_h : h \in {\rm Diff}(\Omega_0) \}$ be the collection of metric spaces.

\begin{definition}
We say that {$X_h \in \mathcal X$ converges to $X_k$ in the Gromov-Hausdorff sense} as $h \rightarrow k$
if for any $\varepsilon >0$ and a bounded set $B_k \subset X_k$,
there is $\delta >0$ such that if $d_{C^1}(h,k)<\delta$
then there is a bounded set $B_h \subset X_h$ satisfying $d_{GH}(B_h,B_k) < \varepsilon$.
\end{definition}

We observe that $X_h$ converges to $X_k$ in the Gromov-Hausdorff sense if $d_{GH}(X_h,X_k) \rightarrow 0$ as $h \rightarrow k$. However the converse is not true in general. Let $S$ be a dynamical system on $X$, i.e., $S:X\times \mathbb R \to X$. For any subset $B$  of $X$, we denote $S|_{B}$ by the restriction of $S$ to $B\times \mathbb R$.

\begin{definition}
Let $S_1$ and $S_2$ be dynamical systems on metric spaces $X$ and $Y$, respectively.
For any bounded sets $B_1 \subset X$ and $B_2 \subset Y$, the Gromov-Hausdorff distance $D_{GH}^T(S_1|_{B_1}, S_2|_{B_2})$ between $S_1|_{B_1}$ and $S_2|_{B_2}$  with respect to $T>0$ is defined by the infimum of $\varepsilon>0$ such that there are maps $i : X \rightarrow Y$ and $j: Y \rightarrow X$, and $\alpha \in Rep_{B_1} (\varepsilon)$ and $\beta \in Rep_{B_2}(\varepsilon)$ with the following properties:
\begin{itemize}
\item[(i)] $i$ and $j$ are $\varepsilon$-isometries on $B_1$ and $B_2$, respectively, satisfying
$$
U_{\varepsilon}(i(B_1)) \cap B_2 = B_2 ~\mbox{and} ~U_{\varepsilon}(j(B_2)) \cap B_1 = B_1,
$$

\item[(ii)] $d_Y (i (S_1(x,\alpha(x,t))) , S_2(i(x),t) ) < \varepsilon$ for $x \in B_1$ and $t \in [-T,T]$, and 
\\ $d_X (j (S_2(y,\beta(y,t))) , S_1(j(y),t) ) < \varepsilon$ for $y \in B_2$ and $t \in [-T,T]$,
\end{itemize}
where $Rep_B(\varepsilon)$ is the collection of continuous maps $\alpha : B \times \mathbb{R} \rightarrow \mathbb{R}$ such that for given $x \in B$, $\alpha (x, .)$ is a homeomorphism on $\mathbb{R}$ with $\left| \dfrac{\alpha(x,t)}{t} - 1 \right| < \varepsilon$ for $t \neq 0$.
\end{definition}

\begin{definition}
Let $\mathcal{DS}= \{ (X_h, S_h): h \in {\rm Diff}(\Omega_0)\}$
be a collection of dynamical systems on metric spaces $X_h$.
We say that a dynamical system $S_k \in \mathcal{DS}$ is 
Gromov-Hausdorff stable if for any  $\varepsilon >0$, $T>0$ and a bounded set $B_k \subset X_k$,
there exists $\delta >0$ such that if $d_{C^1}(h,k)< \delta$ then
there is a bounded set $B_h \subset X_h$ satisfying
$D_{GH}^T(S_h|_{B_h}, S_k|_{B_k}) < \varepsilon$.
\end{definition}

We observe that the Gromov-Hausdorff stability of dynamical systems on the global attractors under perturbations of the domain was studied in \cite{L, LNT}.

Throughout the paper, we assume the following conditions
\begin{equation}\label{condition}
\lambda_{m+1}^0 - \lambda_{m}^0 > 2 \sqrt{2} L_0 ~\mbox{and}~ \lambda_m^0 > L_0~\mbox{for~some}~ m \in \mathbb N.
\end{equation}
Moreover we assume that $m$ is the smallest number satisfying \eqref{condition}, and the inertial manifold $\mathcal M_h$ for equation \eqref{eq:F} means the unique $m$-dimensional inertial manifold for \eqref{eq:F}.
With all the notations in mind, we state the main results of this paper.

\begin{theorem} \label{GH convergence}
The inertial manifold $\mathcal{M}_h$ of equation \eqref{eq:F} converges to $\mathcal{M}_0$ in the Gromov-Hausdorff sense as $h \rightarrow id$.
\end{theorem}

\begin{remark}
We will continue to prove Theorem \ref{GH convergence} by applying the Lyapunov-Perron method to get the Lipschitz map $\Phi_h$ in \eqref{lyapunov}, where $h \in {\rm Diff}(\Omega_0)$. Note that the assumption $\lambda_{m+1}^0 - \lambda_{m}^0 > 2 \sqrt{2} L_0$ is a sharp condition for the construction of $\Phi_h$. In fact, if $\lambda_{m+1}^0 - \lambda_{m}^0 < 2 \sqrt{2} L_0$, then we cannot apply the Lyapunov-Perron technique for the proof of Theorem \ref{GH convergence}.
\end{remark}

Let $S_h(t)$ be the dynamical system on the inertial manifold $\mathcal M_h$ induced by equation \eqref{eq:F}.

\begin{theorem} \label{stability of flows}
The dynamical system $S_0(t)$ on the inertial manifold $\mathcal M_0$ induced by equation  \eqref{eq:f0} is Gromov-Hausdorff stable.
\end{theorem}


\section{Existence of inertial manifolds}

In this section, we analyze the behavior of the Laplace operators with Dirichlet boundary conditions under perturbations of the domain and equation. In particular, we prove that the spectra of $A_0$ behave continuously, and it will be applied to prove Theorems \ref{existence of IM} and \ref{GH convergence}. Note that Arrieta \emph{et al.} assumed  the continuity of the spectra of the Laplace operators with Neumann boundary conditions to prove the continuity of the global attractors (see Definition 2.5 and Theorem 4.6 in \cite{AC}).

\begin{proposition} \label{continuity of spectra}
The spectra of $A_0$ behaves continuously. More precisely, for any fixed $\ell \in \mathbb{N}$ and a sequence $\{h_n\}_{n\in \mathbb N}$ in $\rm{Diff}(\Omega_0)$ with $h_n \rightarrow id$, there exist a subsequence $\{h_k:=h_{n_k}\}_{k \in \mathbb{N}}$ of $\{h_n\}_{{n\in \mathbb N}}$ and a collection of eigenfunctions $\{\xi_1^0, \ldots, \xi_\ell^0 \}$ of $A_0$ with respect to eigenvalues $\{\lambda_1^0, \ldots, \lambda_\ell^0\}$ such that $\lambda_i^{h_{k}} \rightarrow \lambda_i^0$ and $ \phi_i^{h_{k}} \rightarrow  \xi_i^0$ in $L^2(\mathbb{R}^N)$ for all $1 \leq i \leq \ell$.

\end{proposition}

\begin{proof}

Let $\ell \in \mathbb{N}$ be fixed, and take a sequence $\{h_n\}_{n \in \mathbb{N}}$ in $\rm{Diff}(\Omega_0)$ with $h_n \rightarrow id$ as $n \rightarrow \infty$.
Let $E_0: H^1(\Omega_0) \rightarrow H^1 (\mathbb{R}^N)$ be an extension operator, and $R_{h_n}: H^1(\mathbb{R}^N) \rightarrow H^1(\Omega_{h_n})$ be the restriction operator.
For each $1 \leq i \leq \ell$ and $n \in \mathbb{N}$, we consider a map $\xi_i^{h_n}: \Omega_{h_n} \to \mathbb R$ by 
$$
\xi_i^{h_n}(x) = R_{h_n} (E_0 \phi_i^0)(x), ~\forall x \in \Omega_{h_n}.
$$
By the min-max characterization of the eigenvalues, we have
\begin{align}\label{eq:minmax}
\lambda_r^{h_n} \leq \max \{\|\nabla \xi\|_{L^2(\Omega_{h_n})}^2 + o(1): \xi \in {\rm span} \{\xi_1^{h_n}, \ldots, \xi_r^{h_n}\} ~ \hbox{with} ~ \|\xi \|_{L^2(\Omega_{h_n})}=1\}, ~\forall  1 \leq r \leq \ell,
\end{align}
where $t = o(1)$ means that $t \to 0$ as $h_n \to id$. Take a function $\xi = \sum_{i=1}^r a_i \xi_i^{h_n}$ in ${\rm span} \{\xi_1^{h_n}, \ldots, \xi_r^{h_n}\}$ such that the right hand side of \eqref{eq:minmax} has the maximum at $\xi$ with $\| \xi\|_{L^2(\Omega_{h_n})}=1$. Let $\phi = \sum_{i=1}^r a_i \phi_i^0$. Then we see that
\begin{align*}
\|\nabla \xi\|_{L^2(\Omega_{h_n} \cap \Omega_0)}^2 &= \|\nabla \phi \|_{L^2(\Omega_{h_n} \cap \Omega_0)}^2 \le \|\nabla \phi \|_{L^2(\Omega_0)}^2 \le \lambda_r^0 \|\phi \|_{L^2(\Omega_0)}^2 \\
&= \lambda_r^0 \|\phi \|_{L^2(\Omega_0 \cap \Omega_{h_n})}^2 +\lambda_r^0 \|\phi \|_{L^2(\Omega_0 \setminus \Omega_{h_n})}^2 \le \lambda_r^0+  \lambda_r^0~o(1).
\end{align*}
In the last inequality, we have used the fact that 
$$
\|\phi \|_{L^2(\Omega_0 \cap \Omega_{h_n})}= \|\xi \|_{L^2(\Omega_0 \cap \Omega_{h_n})} \leq \|\xi \|_{L^2(\Omega_{h_n})} =1~\mbox{and}~ \|\phi \|_{L^2(\Omega_0 \setminus {\Omega}_{h_n})} = o(1).
$$
Similarly we get $\|\nabla \xi\|_{L^2(\Omega_{h_n} \setminus {\Omega}_0)} = o(1)$. Since
\begin{align*}
\|\nabla \xi\|_{L^2(\Omega_{h_n})}^2 
&= \|\nabla \xi\|_{L^2(\Omega_{h_n} \cap \Omega_0)}^2+ \|\nabla \xi\|_{L^2(\Omega_{h_n} \setminus \Omega_0)}^2,
\end{align*}
we have 
$$
\lambda_r^{h_n} \leq (1+o(1)) \lambda_r^0 + o(1).
$$
Consequently, for each $1 \leq i \leq \ell$, we can take $0<\tau_i \leq \lambda_i^0$ and a subsequence $\{h_k:=h_{n_k}\}_{k \in \mathbb{N}}$ of $\{h_n\}_{n \in \mathbb N}$ such that $\lambda_i^{h_k} \rightarrow \tau_i$ as $k \rightarrow \infty$.
Take a sequence  $\{\varepsilon_n >0 \}_{n \in \mathbb{N}}$ in $\mathbb R$ such that 
$$
\varepsilon_n \rightarrow 0~\mbox{and}~ K_{h_n} := \{ x \in \Omega_0 \mid d(x,\partial \Omega_0) \geq \varepsilon_n \}\subset \Omega_{h_n}, ~\forall n \in \mathbb{N}.
$$

We will complete the proof by demonstrating the following four claims.

{Claim 1.}
We show that for any $1 \leq i \leq \ell$, $\|\phi_i^{h_k} \|_{L^2(\Omega_{h_k} \setminus K_{h_k})} \rightarrow 0$ as $k \rightarrow \infty$.

Let $V$ be an open set in $\mathbb{R}^N$ such that $\Omega_{h_k} \subset V$ for all sufficiently large $k$. By the Sobolev extension theorem, we can take an operator $T_{h_k}: H^1_0 (\Omega_{h_k}) \rightarrow H^1 (\mathbb{R}^N)$ such that
$$
T_{h_k} u = u ~\mbox{on}~ \Omega_{h_k},~ {\rm supp} ~ T_{h_k} u \subset V, ~\mbox{and}~\|T_{h_k} u \|_{H^1 (\mathbb{R}^N)} \leq D \|u \|_{H^1_0 (\Omega_{h_k})}, ~\forall u \in H^1_0 (\Omega_{h_k}),
$$
where $D$ is a constant which is independent on $k$.

Since $\|T_{h_k} \phi_i^{h_k} \|_{H^1(\mathbb{R}^N)}$ is uniformly  bounded on $k$, there are $\phi_0 \in L^2(\mathbb{R}^N)$ and a subsequence of $\{ T_{h_k} \phi_i^{h_k} \}_{k \in \mathbb N}$, still denoted by $\{ T_{h_k} \phi_i^{h_k} \}_{k \in \mathbb N}$, such that $T_{h_k} \phi_i^{h_k} \rightarrow \phi_0$ in $L^2(\mathbb{R}^N)$ as $k \rightarrow \infty$. We have
$$
\|\phi_i^{h_k} \|_{L^2(\Omega_{h_k} \setminus K_{h_k})} \leq \| T_{h_k} \phi_i^{h_k} - \phi_0\|_{L^2(V)} + \| \phi_0\|_{L^2(\Omega_{h_k} \setminus K_{h_k})}
$$

\noindent
Since $|\Omega_h \setminus K_h| \rightarrow 0$ as $h \rightarrow id$, we derive that $\|\phi_i^{h_k} \|_{L^2(\Omega_{h_k} \setminus K_{h_k})} \rightarrow 0$ as $k \rightarrow \infty$. This completes the proof of Claim 1.

For any $1 \leq i \leq \ell$, we will consider the limit of the sequence $\{\phi_i^{h_k}\}_{k \in \mathbb{N}}$ in Claim 1. By the induction process, for each $n \in \mathbb{N}$, there exist $\xi_i^{0,n} \in L^2(K_{h_n})$ and a subsequence $\{\phi_i^{h_{k,n}}\}_{k \in \mathbb{N}}$ of  $\{\phi_i^{h_{k,n-1}}\}_{k \in \mathbb{N}}$ which converges to $\xi_i^{0,n}$ strongly in $L^2(K_{h_n})$ (and weakly in $H^1(K_{h_n})$) as $k \rightarrow \infty$, where $\phi_i^{h_{k,0}}: = \phi_i^{h_k}$ for all $i = 1, \ldots, \ell$ and $k \in \mathbb N$.
By the Cantor diagonal argument, we assume that there is a subsequence $ \{\phi_i^{h_k}:= \phi_i^{h_{k,k}}\}_{k \in \mathbb{N}}$ of  $ \{\phi_i^{h_{k,n}}\}_{k,n \in \mathbb{N}}$ such that $\{\phi_i^{h_k}\}_{k\in \mathbb N}$ converges to $\xi_i^{0,n}$ strongly in $L^2(K_{h_n})$ (and weakly in $H^1(K_{h_n})$) as $k \rightarrow \infty$.
Since
$$
\|\xi_i^{0,n} - \xi_i^{0,n+1}\|_{L^2(K_{h_n})}
\leq \| \xi_i^{0,n} - \phi_i^{h_k}\|_{L^2(K_{h_n})} + \|\phi_i^{h_k} - \xi_i^{0,n+1} \|_{L^2(K_{h_{n+1}})} \to 0 ~\mbox{as}~k \to \infty,
$$
we see that $\xi_i^{0,n} = \xi_i^{0,n+1}$ almost everywhere on $K_{h_n}$.
Define a map $\xi_i^0 : \Omega_0 \rightarrow \mathbb{R}$ by $\xi_i^0 (x) = \xi_i^{0,n}(x)$, where $n$ is the natural number satisfying $x \in K_{h_n} \setminus K_{h_{n-1}}$.

Now we show that $\xi_i^0 \in H^1(\Omega_0)$. We first consider an extension operator $E_{h_n}: H^1(K_{h_n}) \to H^1(\mathbb R^N)$. Then $\{ E_{h_n} \xi_i^{0,n} \}_{n \in \mathbb{N}}$  is a sequence in $H^1(\Omega_0)$. Since 
\begin{align*}
\|E_{h_n} \xi_i^{0,n} \|_{H^1(\Omega_0)}
&\leq D \|\xi_i^{0,n} \|_{H^1(K_{h_n})}
\leq D \lim_{k \rightarrow \infty} \|\phi_i^{h_k} \|_{H^1(K_{h_n})} \\
&\leq D  \lim_{k \rightarrow \infty} \big(1+ (\lambda_i^{h_k})^{1/2}\big) \|\phi_i^{h_k} \|_{L^2(\Omega_{h_k})} \leq D\big(1+ (\lambda_n^0)^{1/2}\big),~\forall n \in \mathbb{N},
\end{align*}
where $D$ is a positive constant independent of $h_k$, we see that $\{ E_{h_n} \xi_i^{0,n} \}_{n \in \mathbb{N}}$ is bounded in $H^1(\Omega_0)$.
Hence there are $\tilde{\xi}_i^0 \in H^1(\Omega_0)$ and
a subsequence of $\{ E_{h_n} \xi_i^{0,n} \}_{n \in \mathbb{N}}$, still denoted by $\{ E_{h_n} \xi_i^{0,n} \}_{n \in \mathbb{N}}$, which converges to $\tilde{\xi}_i^0$ in $L^2(\Omega_0)$.
Moreover, for each $K_{h_n}$, we have
$$
\|\xi_i^0 - \tilde{\xi}_i^0 \|_{L^2(K_{h_n})} \le \|\xi_i^{0} - \xi_i^{0,n}\|_{L^2(K_{h_n})} + \|\xi_i^{0,n} - \tilde{\xi}_i^{0}\|_{L^2(K_{h_n})} \to 0 ~\mbox{as}~ n \to \infty,
$$
and so $\xi_i^0 = \tilde{\xi}_i^0$ almost everywhere in $K_{h_n}$ for all $n \in \mathbb{N}$.
Since $\bigcup_{n \in \mathbb{N}} K_{h_n} = \Omega_0$, we see that $\xi_i^0 = \tilde{\xi}_i^0 \in H^1(\Omega_0)$.

{Claim 2.}
 $\phi_i^{h_k} \rightarrow \xi_i^0$ in $L^2(\mathbb{R}^N)$ as $k \rightarrow \infty$.

By the construction of $\xi_i^0$ and Claim 1, we see that for any $\varepsilon >0$, there is $k_0 \in \mathbb{N}$ such that 
$$
\|\phi_i^{h_k} \|_{L^2(K_{h_k} \setminus K_{h_{k_0}})} \le \varepsilon/4,~
\|\xi_i^0 \|_{L^2(\Omega_0 \setminus K_{h_{k_0}})} \leq \varepsilon /4,
$$
$$
\|\phi_i^{h_k} \|_{L^2(\Omega_{h_k} \setminus K_{h_k})} \leq \varepsilon/4~\mbox{and}~ \|\phi_i^{h_k} - \xi_i^0 \|_{L^2(K_{h_{k_0}})} < \varepsilon, ~\forall k \geq k_0.
$$
Since
\begin{align*}
\|\phi_i^{h_k} - \xi_i^0 \|_{L^2(\mathbb{R}^N)}^2
= \|\phi_i^{h_k} - \xi_i^0 \|_{L^2(K_{k_0})}^2 
+\|\phi_i^{h_k} - \xi_i^0 \|_{L^2(\mathbb{R}^N \setminus K_{k_0})}^2
\end{align*}
and
\begin{align*}
\|\phi_i^{h_k} - \xi_i^0\|_{L^2(\mathbb{R}^N \setminus K_{k_0})}
\leq \|\phi_i^{h_k} \|_{L^2(\Omega_{h_k} \setminus K_{h_k})}
+ \|\phi_i^{h_k} \|_{L^2(K_{h_k} \setminus K_{h_{k_0}})}
+ \|\xi_i^0 \|_{L^2(\Omega_0 \setminus K_{h_{k_0}})} < \varepsilon,
\end{align*}
we have that $\|\phi_{h_k} - \xi_0 \|_{L^2(\mathbb{R}^N)}^2 \leq 2\varepsilon^2$. Since $\varepsilon$ is arbitrary, we get $\|\phi_{h_k} - \xi_0 \|_{L^2(\mathbb{R}^N)}^2 \rightarrow 0$ as $k \rightarrow \infty$.

{Claim 3.} $\lambda_i^{h_k} \rightarrow \tau_i$ as $k \rightarrow \infty$.

For any $k_0 \in \mathbb N$ and $k \ge k_0$ and $\xi \in C_c^1(K_{h_{k_0}})$, we first note that
\begin{align*}
\left|\int_{\Omega_{h_k}} \nabla \phi_i^{h_k} \nabla \xi - \int_{\Omega_0} \nabla \xi_i^0 \nabla \xi \right|
&\leq \left|\int_{K_{h_{k_0}}} (\nabla \phi_i^{h_k} - \nabla \xi_i^0) \nabla \xi \right| \\
&\;\;+ \int_{\Omega_{h_k}\setminus K_{h_{k_0}} } \left| \nabla \phi_i^{h_k} \right|  \left|\nabla \xi \right|
+ \int_{\Omega_0 \setminus K_{h_{k_0}}} \left|\nabla \xi_i^0 \right|  \left|\nabla \xi \right|.
\end{align*}
Since $\phi_i^{h_k} \rightarrow \xi_i^0$ weakly in $H^1(K_{h_{k_0}})$ and $\nabla \xi = 0$ outside $K_{h_{k_0}}$, we get

$$
\left|\int_{\Omega_{h_k}} \nabla \phi_i^{h_k} \nabla \xi - \int_{\Omega_0} \nabla \xi_i^0 \nabla \xi \right| \rightarrow 0
\hbox{ as } h_k \rightarrow id.
$$

On the other hand, we have

$$
\int_{\Omega_{h_k}} \nabla \phi_i^{h_k} \nabla \xi = \int_{\Omega_{h_k}}\lambda_i^{h_k} \phi_i^{h_k} \xi \rightarrow \int_{\Omega_{0}}\tau_i \xi_i^0 \xi \hbox{ as } h_k \rightarrow id.
$$
Since $\bigcup_{k \in \mathbb N} C_c^1(K_{h_k})$ is dense in $H_0^1(\Omega_0)$, we get $-\Delta \xi_i^0 = \tau_i \xi_i^0$.
Therefore $\{\tau_i\}_{i \in \mathbb N}$ and $\{\xi_i^0\}_{i \in \mathbb N}$ are eigenvalues and the corresponding eigenfunctions of $A_0$, respectively, and so we have
$$
\lambda_i^{h_k} \rightarrow \tau_i~\mbox{and}~ \phi_i^{h_k} \rightarrow \xi_i^0~\mbox{as}~k \rightarrow \infty.
$$
This completes the proof of Claim 3.

{Claim 4.} $\tau_i = \lambda_i^0$ for all $1 \leq i \leq \ell$, and $\{\xi_i^0\}_{i=1}^\ell$ is orthonormal in $L^2(\Omega_0)$.

For $1 \leq i, j \leq \ell$, by Claim 2, we get

\begin{align*}
&\left| \int_{\Omega_{h_k}} \phi_i^{h_k} \phi_j^{h_k} - \int_{\Omega_0} \xi_i^0 \xi_j^0  \right|
= \left| \int_{\mathbb{R}^N} \phi_i^{h_k} \phi_j^{h_k} - \xi_i^0 \xi_j^0   \right| \\
&\hspace{20mm}\leq \int_{\mathbb{R}^N} |\phi_i^{h_k}| |\phi_j^{h_k} - \xi_j^0|
+ \int_{\mathbb{R}^N} |\phi_i^{h_k} - \xi_i^0| |\xi_j^0| \\
&\hspace{20mm}\leq \|\phi_i^{h_k}\|_{L^2(\Omega_{h_k}) }  \|\phi_j^{h_k} - \xi_j^0\|_{L^2(\mathbb{R}^N)}   + \|\phi_i^{h_k} - \xi_i^0 \|_{L^2(\mathbb{R}^N)} \|\xi_j^0 \|_{L^2(\Omega_0)} \rightarrow 0~\mbox{as}~ k \rightarrow \infty.
\end{align*}
Consequently, we have
$$
(\phi_i^{h_k}, \phi_j^{h_k})_{L^2(\Omega_{h_k})} \rightarrow (\xi_i^0, \xi_j^0)_{L^2(\Omega_0)}~\mbox{as}~k \rightarrow \infty,
$$
and so $\{\xi_i^0\}_{i=1}^{\ell}$ is orthonormal in $L^2(\Omega_0)$.

Let $a$ be the multiplicity of $\lambda_1^0$.
Since $\lambda_1^0$ is the smallest eigenvalue and $\tau_j \leq \lambda_j^0$, we have $\tau_j = \lambda_1^0$ for all $j\leq a$.
Let $\xi = \sum_{i=1}^\infty p_i \phi_i^0$ be a function satisfying $A_0 \xi = \lambda_1^0 \xi$. Since
$$
A_0\xi=\sum_{i=1}^\infty \lambda_i^0 p_i \phi_i^0 \quad \mbox{and} \quad \lambda_1^0\xi=\sum_{i=1}^\infty \lambda_1^0 p_i \phi_i^0,
$$
we have $p_i =0$ for all $i>a$. This means that $\xi \in {\rm span} \{\phi_1^0, \ldots, \phi_a^0\}$ and so
$$
{\rm span} \{\phi_1^0, \ldots, \phi_a^0\} = {\rm span} \{\xi_1^0 \ldots, \xi_a^0\}.
$$
Suppose that $\tau_{a+1} \neq \lambda_{a+1}^0$. Since $A_0 \xi_{a+1}^0 = \lambda_1^0 \xi_{a+1}^0$, we have 
$$
\xi_{a+1}^0 \in {\rm span} \{\phi_1^0, \ldots, \phi_a^0\} = {\rm span} \{\xi_1^0 \ldots, \xi_a^0\} \quad \mbox{and} \quad \tau_{a+1} = \lambda_1^0.
$$
This contradicts to the orthonormality of $\{\xi_i^0 \}_{i=1}^{a+1}$, and so we get $\tau_{a+1} = \lambda_{a+1}^0$.

Continuing this process, we derive that $\tau_i = \lambda_i^0$ for all $1 \leq i \leq \ell$, and so completes the proof of the proposition.
\end{proof}

For any $h\in {\rm Diff}(\Omega_0)$, we let 
\begin{align}
L_h = \sup_{s \in \mathbb R}| f'_h(s)| \quad \mbox{and} \quad
L_0 = \sup_{s \in \mathbb R}| f'_0(s)|. \label{Lipscons}
\end{align}
It is clear that $L_h$ and $L_0$ are Lipschitz constants of the nonlinear terms $F_h$ and $F_0$, respectively, such that $L_h \rightarrow L_0$ as $h \rightarrow id$.

Define a map $j_{h} :L^2(\Omega_0) \rightarrow L^2(\Omega_{h})$ by 
$$
j_{h}(u) := u \circ {h}^{-1}, ~\forall u \in L^2(\Omega_0) .
$$
Then we see that $j_h$  is an isomorphism, and $\|j_{h} \| \rightarrow 1$ as $h \to id$.
Here $\|j_h \|=\|j_h \|_{L^\infty(L^2(\Omega_0),L^2(\Omega_h))}$.
Hence we may assume that $\|j_{h} \| < 2$ for all $h  \in{\rm Diff}(\Omega_0)$.

Now we prove the existence of inertial manifold of \eqref{eq:F} under perturbations of the domain and equation.

\begin{proof}[\bf Proof of Theorem \ref{existence of IM}]
Let $\eta = ({\lambda_{m+1}^0 - \lambda_m^0 - 2 \sqrt{2} L_0})/{2}$.
We first show that there is $\delta >0$ such that if $d_{C^1} (h, id) < \delta$, then
$$
\left| \lambda_m^0 - \lambda_m^h \right| < \eta/2 ~\mbox{and}~ \left| \lambda_{m+1}^0 - \lambda_{m+1}^h \right| < \eta/2.
$$
Suppose not. Then for any $n \in \mathbb{N}$, there is $h_n \in {\rm Diff}(\Omega_0)$ with $d_{C^1} (h_n, id) \leq 1/n$ such that
$$
\left| \lambda_m^0 - \lambda_m^{h_n} \right| \geq \eta/2~\mbox{or}~\left| \lambda_{m+1}^0 - \lambda_{m+1}^{h_n} \right| \geq \eta/2.
$$
Without loss of generality, we may assume that $\left| \lambda_m^0 - \lambda_m^{h_n} \right| \geq \eta/2$ for all $n \in \mathbb{N}$.

On the other hand, by Proposition \ref{continuity of spectra}, there is a subsequence $\{h_{n_k}\}_{k \in \mathbb{N}}$ of $\{h_n\}_{n\in \mathbb N}$ such that $\lambda_m^{h_{n_k}} \rightarrow \lambda_m^0$ as $k \rightarrow \infty$.
The contradiction shows that there is $\delta >0$ such that if $d_{C^1}(h,id) < \delta$, then 
$$
\lambda_{m+1}^h - \lambda_m^h \geq  (\lambda_{m+1}^0 - \lambda_{m}^0) -|\lambda_{m+1}^0 - \lambda_{m+1}^h| - |\lambda_m^h - \lambda_m^0| > 2 \sqrt{2} L_0 + \eta.
$$
Since $L_h \rightarrow L_0$ as $h \rightarrow id$, we see that $\lambda_{m+1}^h - \lambda_m^h  > 2 \sqrt{2} L_h$ if $h$ is sufficiently $C^1$-close to $id$. Hence equation \eqref{eq:F} admits the $m$ dimensional inertial manifold $\mathcal{M}_h$ in $L^2(\Omega_h)$ which can be presented by the graph of a Lipschitz map $\Psi_h$.
\end{proof}

\section{Proof of Theorem \ref{GH convergence}}
Suppose not. Then there are $\varepsilon >0$ and a bounded set $B_0 \subset \mathcal{M}_0$ such that for any $n \in \mathbb{N}$, 
there is $h_n \in {\rm Diff}(\Omega_0)$ with $d_{C^1}(h_n,id)< 1/n$
such that for any bounded set $B_{h_n} \subset \mathcal{M}_{h_n}$,
$d_{GH}(B_{h_n},B_0) \geq \varepsilon$.
For each $n \in \mathbb{N}$, let $\{\lambda_1^{h_n}, \ldots, \lambda_m^{h_n} \}$ and $\{\phi_1^{h_n}, \ldots, \phi_m^{h_n}\}$ be the first $m$ eigenvalues and corresponding  $m$ eigenfunctions of $A_{h_n}$, respectively.
By Proposition \ref{continuity of spectra}, there are $m$ eigenfunctions, denoted by $\{\phi_1^0, \ldots, \phi_m^0\}$, with respect to the first $m$ eigenvalues $\{\lambda_1^0, \ldots, \lambda_m^0\}$ of $A_0$, and a subsequence of $\{h_n\}_{n \in \mathbb N}$, still denoted by $\{h_n\}_{n \in \mathbb N}$, such that
$\lambda_i^{h_n} \rightarrow \lambda_i^0$ and
$\phi_i^{h_n} \rightarrow \phi_i^0$ in $L^2(\mathbb{R}^N)$ as $n \rightarrow \infty$
for all $1 \leq i \leq m$.
We assume that $|\lambda_i^{h_n} - \lambda_i^0|<1$ for all $n \in \mathbb{N}$ and $1 \leq i \leq m$.
For any $h \in {\rm Diff}(\Omega_0)$, define a map $\psi_h: P_m^{h}L^2(\Omega_h) \to \mathbb R^m$ by
$$
\psi_h\left(\sum_{i=1}^ma_i\phi_i^h\right) = (a_1.\ldots, a_m),\quad a_i \in \mathbb R.
$$
For each $n \in \mathbb{N}$, we denote by $B_{h_n}$ the collection of $u_{h_n} \in \mathcal{M}_{h_n}$ such that $\psi_{h_n} P_m^{h_n}u_{h_n} = \psi_0 P_m^0 u_0$ for some $u_0 \in B_0$.

We will complete the proof of Theorem \ref{GH convergence} by showing that $d_{GH}(B_{h_n}, B_0) < \varepsilon$ for all sufficiently large $n \in \mathbb{N}$. For this, we need several lemmas.

\begin{lemma} \label{alpha}
For any fixed $1 \leq i \leq m$, we have
$$
\|j_{h_n} \phi_i^0 - \phi_i^{h_n} \|_{L^2(\Omega_{h_n})} \rightarrow 0 \hbox{ as } n \rightarrow \infty.
$$
\end{lemma}

\begin{proof}
For any fixed $1 \leq i \leq m$, we have
\begin{align*}
\|j_{h_n} \phi_i^0 - \phi_i^{h_n} \|_{L^2(\Omega_{h_n})}
\leq \|j_{h_n} \phi_i^0 - \phi_i^0 \|_{L^2(\mathbb{R}^N)} + \|\phi_i^0 - \phi_i^{h_n} \|_{L^2(\mathbb{R}^N)}
\end{align*}
By Proposition \ref{continuity of spectra}, we have that 
$$
\|\phi_i^0 - \phi_i^{h_n} \|_{L^2(\mathbb{R}^N)} \to 0~\mbox{as}~ n \rightarrow \infty.
$$
So it is sufficient to prove that $\|j_{h_n} \phi_i^0 - \phi_i^0 \|_{L^2(\mathbb{R}^N)} \rightarrow 0$ as $n \rightarrow \infty$.

Take a neighborhood $V$ of $\Omega_0$ such that $\Omega_{h_n} \subset V$ for all $n \in \mathbb{N}$.
Then we have
\begin{align*}
\|j_{h_n} \phi_i^0 - \phi_i^0 \|_{L^2(V)}
&\leq \|j_{h_n} \phi_i^0 - \phi_i^0 \|_{L^2(\Omega_0 \cap \Omega_{h_n})}
+ \|j_{h_n} \phi_i^0 \|_{L^2(\Omega_{h_n} \setminus \ \overline{\Omega}_0)}
+ \|\phi_i^0 \|_{L^2(\Omega_{0} \setminus \overline{\Omega}_{h_n})}\\
&:= I_n + II_n + III_n.
\end{align*}
Let $E : H^1(\Omega_0) \rightarrow H^1(\mathbb{R}^N)$ be an extension operator such that
$$
E u = u~\mbox{on}~ \Omega_0, ~{\rm supp} ~ E u \subset V, ~\mbox{and}~ \|E u \|_{H^1(\mathbb{R}^N)} \le C \|u \|_{H^1(\Omega_0)}, ~\forall u \in H^1(\Omega_0).
$$
Since $C^1(V)$ is dense in $H^1(V)$, we can take a sequence $\chi_k$ in $C^1(V)$ such that $\chi_k \rightarrow E\phi_i^0$ in $H^1(V)$ as $k \rightarrow \infty$. Note that

\begin{align}
\|j_{h_n} \phi_i^0 - \phi_i^0 \|_{L^2(\Omega_0 \cap \Omega_{h_n})}
\leq & ~ \|j_{h_n} \phi_i^0 - j_{h_n} \chi_k \|_{L^2(\Omega_0 \cap \Omega_{h_n})}
+ \|j_{h_n} \chi_k - \chi_k \|_{L^2(\Omega_0 \cap \Omega_{h_n})} \nonumber \\
&+ \|\chi_k - \phi_i^0 \|_{L^2(\Omega_0 \cap \Omega_{h_n})}. \label{j_hn}
\end{align}
Since $E \phi_i^0 = \phi_i^0$ on $\Omega_0$, the first and last terms in the right hand side of \eqref{j_hn} tend to $0$ as $k \rightarrow \infty$.
For the second term, we get

\begin{align*}
\int_{\Omega_0 \cap \Omega_{h_n}} |\chi_k (h_{n}^{-1} x) - \chi_k (x)|^2 dx
&\leq d_{C^1}(h_n^{-1},id) \int_0^1 \int_{\Omega_0 \cap \Omega_{h_n}} |D\chi_m ((1-t)x + th_n^{-1}(x))|^2 dx dt\\
&\leq 2 d_{C^1} ({h_n^{-1}},id) \int_V |D \chi_k (x)|^2 dx.
\end{align*}
Since
\begin{align*}
\int_{V} |D \chi_k (x)|^2 dx \rightarrow \int_V |D E \phi_i^0 (x)|^2 dx \leq C \int_{\Omega_0} |D\phi_i^0(x)|^2dx,
\end{align*}
we see that $\int_V |D \chi_k (x)|^2 dx$ is uniformly bounded on $k$.
Hence we obtain
$$
\|j_{h_n} \chi_k - \chi_k \|_{L^2(\Omega_0 \cap \Omega_{h_n})} \rightarrow 0 \hbox{ as } n \rightarrow \infty.
$$
This implies that $I_n \rightarrow 0$ as $n \rightarrow \infty$.

Next, we suppose that $II_n$ does not converge to $0$ as $n \rightarrow \infty$. Then there are $\delta >0$ and a subsequence of $\{h_n\}_{n\in \mathbb N}$, still denoted by $\{h_n\}_{n\in \mathbb N}$, such that $II_{n} \geq \delta$ for all $n \in \mathbb{N}$.
Since $\| j_{h_n} \phi_i^0 \|_{H^1 (V)}$ is uniformly bounded on $n$, there exists a subsequence of $\{ j_{h_n} \phi_i^0 \}_{n \in \mathbb{N}}$, still denoted by $\{j_{h_n} \phi_i^0\}_{n \in \mathbb{N}}$, and $u_0 \in H^1(V)$ such that $j_{h_n} \phi_i^0 \rightarrow u_0$ in $L^2(V)$. Then we have
\begin{align*}
\|j_{h_n} \phi_i^0 \|_{L^2(\Omega_{h_n} \setminus {\Omega}_0)}
\leq  \|j_{h_n} \phi_i^0 - u_0 \|_{L^2(\Omega_{h_n} \setminus {\Omega}_0)} + \|u_0 \|_{L^2(\Omega_{h_n} \setminus {\Omega}_0)}.
\end{align*}
By the fact $|\Omega_{h_n} \setminus {\Omega}_0| \rightarrow 0$, we get $\|j_{h_n} \phi_i^0 \|_{L^2(\Omega_{h_n} \setminus {\Omega}_0)} \rightarrow 0$ as $n \rightarrow \infty$.
The contradiction shows that $II_n \rightarrow 0$ as $n \rightarrow \infty$.

Since $|\Omega_{0} \setminus {\Omega}_{h_n}| \rightarrow 0$ as $n \rightarrow \infty$, we see that $III_{n} \rightarrow 0$.
Consequently, we have $\|j_{h_n} \phi_i^0 - \phi_i^0 \|_{L^2(V)} \rightarrow 0$ as $n \rightarrow \infty$, and so completes the proof.

\end{proof}

Let $\Psi_h:P_m^{h}L^2(\Omega_h) \to Q_m^{h}L^2(\Omega_h)$ be the Lipschitz map whose graph is the inertial manifold $\mathcal M_h$ in Theorem \ref{existence of IM}. We may assume  Lip$\Psi_h \le 1$ (see the proof of Theorem 1 in \cite{Ro}). If we let $\Phi_h = \Psi_h \circ \psi_h^{-1}$, then $\mathcal M_h$ can be considered as the graph of $\Phi_h$ with Lip$\Phi_h \le 1$.

With the notations, we have the following lemma.

\begin{lemma} \label{compare psi}
For any $p_0 \in P_m^0 L^2(\Omega_0)$ and $p_{n} \in P_m^{h_n}L^2(\Omega_{h_n})$,
$$
\left|\psi_{h_n} p_{n} - \psi_0 p_0 \right|_{\mathbb{R}^m}
\leq  \alpha(h_n) \sum_{i=1}^m |a_i| + \|j_{h_n} p_0 -p_{n} \|_{L^2(\Omega_{h_n})},
$$
where $\alpha(h_n) = \sup \{\|j_{h_n} \phi_i^0 - \phi_i^{h_n} \|_{L^2(\Omega_{h_n})} : i =1, \ldots, m\}$, and  $p_0 = \sum_{i=1}^m a_i \phi_i^0$.
\end{lemma}

\begin{proof}
For any $p_0 \in P_m^0 L^2(\Omega_0)$ and $p_{n} \in P_m^{h_n}L^2(\Omega_{h_n})$, there are $a_i, b_i \in \mathbb R$ such that $p_0 = \sum_{i=1}^m a_i \phi_i^0$ and $ p_{n} = \sum_{i=1}^m b_i \phi_i^{h_n}$.
Then we have
\begin{align*}
\left|\psi_{h_n} p_{n} - \psi_0 p_0 \right|_{\mathbb{R}^m}
&= \left\| \sum_{i=1}^m (a_i-b_i) \phi_i^{h_n} \right\|_{L^2(\Omega_{h_n})}\\
&\leq \left\| \sum_{i=1}^m a_i (\phi_i^{h_n} - j_{h_n} \phi_i^0) \right\|_{L^2(\Omega_{h_n})}
+ \left\| j_{h_n} \sum_{i=1}^m a_i  \phi_i^0 - \sum_{i=1}^m b_i \phi_i^{h_n} \right\|_{L^2(\Omega_{h_n})}  \\
&\leq \alpha(h_n) \sum_{i=1}^m |a_i| + \|j_{h_n} p_0 -p_{n} \|_{L^2(\Omega_{h_n})}.
\end{align*}
\end{proof}

By Proposition \ref{continuity of spectra}, we can take a constant $r>0$ such that $\lambda_1^0, \lambda_1^{h_n} > r$ for all $n \in \mathbb{N}$.
For any $n \in \mathbb{N}$ and $T>0$, we denote by
$$
\gamma_{h_n} (T) = \sup \{|e^{-\lambda_i^{h_n}t} - e^{-\lambda_i^0 t}| : 1 \leq i \leq m, -T \leq t \leq T\},
$$
$$
\rho(h_n) = \|F_{h_n} (j_{h_n} u) - j_{h_n} F_0(u)\|_{L^{\infty}(L^2(\Omega_0),L^2(\Omega_{h_n}))}.
$$
Then we observe that $\gamma_{h_n}(T) \to 0$ and $\rho(h_n) \rightarrow 0$ as $n \rightarrow \infty$.
For any $p \in \mathbb{R}^m$ and a bounded set $B \subset \mathbb{R}^m$, we denote by
$$
\beta_{h_n} (p) = \|\Phi_{h_n} (p) - j_{h_n} \Phi_0 (p) \|_{L^2(\Omega_{h_n})} ~ \hbox{and} ~ \beta_{h_n}(B) = \sup \{\beta_{h_n}(p): p \in B\}.
$$
Let $p_0(t)$ and $p_{n}(t)$ be the solutions of

\begin{equation} \label{eq:p_0}
\dfrac{dp_0}{dt} + A_0 p_0 = P_m^0 F_0 (p_0 + \Phi_0 (\psi_0 p_0)), ~\mbox{and}
\end{equation}
\begin{equation} \label{eq:p_h}
\dfrac{dp_{n}}{dt}  + A_{h_n} p_{n} = P_m^{h_n}F_{h_n} (p_n + \Phi_{h_n} (\psi_{h_n} p_n ))
\end{equation}
with initial conditions $p_0(0) = \psi_0^{-1} p$ and $p_n (0) = \psi_{h_n}^{-1} p$, respectively, for some $p \in \mathbb R^m$. With these notations, we have the following estimates.

\begin{lemma} \label{compare p(t)}
For any $T>0$ and a bounded subset $B$ of $\mathcal M_0$, there exists $C>0$ such that for any $t \in [-T,0]$
\begin{align*}
\|p_{n}(t) - j_{h_n} p_0(t) \|_{L^2(\Omega_{h_n})}
\leq& \Bigg( e^{(\lambda_m^0 +1) t} C \gamma_{h_n}(T) 
+ C \alpha(h_n)
+ \dfrac{1}{\lambda_m^0 +1} L_{h_n} \beta_{h_n} (\psi_0 B_{-T})  + \dfrac{1}{\lambda_m^0 +1} \rho({h_n})\\
&+ 2T e^{(\lambda_m^0 +1) t} C\gamma_{h_n} (T)
+ \dfrac{ C(2 +L_{h_n})}{\lambda_m^0 +1} \alpha(h_n) \Bigg) e^{(2L_{h_n}-\lambda_m^0 -1) t}, ~\forall n\in \mathbb N,
\end{align*}
and for any $t \in [0,T]$
\begin{align*}
\|p_{n}(t) -j_{h_n} p_0(t) \|_{L^2(\Omega_{h_n})}
&\leq   \Bigg( e^{r t} C \gamma_{h_n}(T)
+ C \alpha(h_n) 
+ e^{rt} \dfrac{L_{h_n}}{r} \beta_{h_n} (\psi_0 B_T) + e^{rt} \dfrac{\rho(h_n)}{r}\\
& \quad + 2T e^{r t} C\gamma_{h_n}(T) 
+ e^{rt}\dfrac{ C(2+L_{h_n})}{r} \alpha({h_n}) \Bigg) e^{(2L_{h_n}-r) t}, ~\forall n\in \mathbb N,
\end{align*}
where $p_0(t)$ and $p_n(t)$ are the solutions of  \eqref{eq:p_0} and \eqref{eq:p_h}, respectively, such that $p_0(0) \in P_m^0B$ and $\psi_0p_0(0) = \psi_{h_n}p_n(0)$, and $B_{-T} = \{p_0(t) : t \in [-T,0]\}$ and $B_T = \{p_0(t) : t \in [0,T]\}$.
\end{lemma}

\begin{proof}
Let $T>0$ be arbitrary, $B$ a bounded subset of $\mathcal M_0$, and denote $\hat{B}_{-T} = S_0 (B,[-T,0])$ and $\hat{B}_T = S_0 (B,[0,T])$.
 Let  $p_0(t)$ and $p_n(t)$ be the solutions of  \eqref{eq:p_0} and \eqref{eq:p_h}, respectively, such that $p_0(0) \in P_m^0B$ and $\psi_0 p_0(0) = \psi_{h_n}p_n(0)$. By the variation of constant formula for \eqref{eq:p_0} and \eqref{eq:p_h}, we have
\begin{align} 
p_{n}(t) - j_{h_n} p_0(t)
=& ~ e^{-A_{h_n} t} p_{n}(0) - j_{h_n} e^{-A_0 t} p_0(0) \notag \\
&+ \int_0^t e^{-A_{h_n} (t-s)} ( P_m^{h_n} F_{h_n}(p_{n} + \Phi_{h_n} (\psi_{h_n} p_n )) - P_m^{h_n} j_{h_n} F_0 (p_0 + \Phi_0 (\psi_0 p_0)) )ds \notag \\
&+ \int_0^t (e^{-A_{h_n}(t-s)}P_m^{h_n} j_{h_n} -  j_{h_n} e^{-A_0 (t-s)} P_m^0) F_0 ( p_0 + \Phi_0 (\psi_0 p_0)) ds \notag\\
&:= ~ I+ II + III,~\forall t\in[-T,T]. \label{p-jp}
\end{align}
Since $F_0 (\hat{B}_{-T})$ and $F_0 (\hat{B}_T)$ are bounded in $L^2(\Omega_0)$, there exists a constant $C>0$ such that for any $u = \sum_{i=1}^m a_i \phi_i^0$ in $P_m^0  \hat{B}_{-T} \cup P_m^0 \hat{B}_T$ and $v =\sum_{i=1}^m b_i \phi_i^0$ in $P_m^0F_0 (\hat{B}_{-T}) \cup P_m^0 F_0(\hat{B}_T)$, we have  $\sum_{i=1}^m |a_i| < C$ and $\sum_{i=1}^m |b_i| < C$.

Step 1. We first, we estimate $I$ for $t \in [-T,0]$. We write $p_0(0) = \sum_{i=1}^m a_i\phi_i^0$. Since
\begin{align*}
I
&= e^{-A_{h_n} t} \sum_{i=1}^m a_i \phi_i^{h_n} - j_{h_n} e^{-A_0 t} \sum_{i=1}^m a_i\phi_i^0 \notag \\
&= \sum_{i=1}^m a_i( e^{-\lambda_i^{h_n} t} - e^{-\lambda_i^0 t} ) \phi_i^{h_n} + \sum_{i=1}^m a_i e^{-\lambda_i^0 t} (\phi_i^{h_n} - j_{h_n} \phi_i^0),
\end{align*}
we have
\begin{align} \label{I1}
\| I \|_{L^2(\Omega_{h_n})}
&\leq \left\|\sum_{i=1}^m ( e^{-\lambda_i^{h_n} t} - e^{-\lambda_i^0 t} )a_i \phi_i^{h_n} \right\|_{L^2(\Omega_{h_n})}
+ \sum_{i=1}^m |a_i|  e^{-\lambda_i^0 t} \|\phi_i^{h_n} - j_{h_n} \phi_i^0\|_{L^2(\Omega_{h_n})} \notag \\
&\leq \gamma_{h_n}(T) \sum_{i=1}^m |a_i|
+ e^{-(\lambda_m^0 +1)t} \alpha({h_n}) \left(\sum_{i=1}^m |a_i|\right) \\
&\leq C\gamma_{h_n} (T) + C e^{-(\lambda_m^0 +1)t} \alpha({h_n}).
\end{align}

Step 2. We estimate $II$ for $t \in [-T,0]$. For this, we first consider the following.
\begin{align*}
F_{h_n} (p_{n} + &\Phi_{h_n} \psi_{h_n} p_{n}) - j_{h_n} F_0 (p_0 + \Phi_0 \psi_0 p_0) \\
=& ~ F_{h_n} (p_{n} + \Phi_{h_n} \psi_{h_n} p_{n}) - F_{h_n} (j_{h_n} p_0 + \Phi_{h_n} \psi_{h_n} p_{n}) \\
&+ F_{h_n} (j_{h_n} p_0 + \Phi_{h_n} \psi_{h_n} p_{n}) - F_{h_n} (j_{h_n} p_0 + \Phi_{h_n}  \psi_0 p_0) \notag \\
&+ F_{h_n} (j_{h_n} p_0 + \Phi_{h_n}  \psi_0 p_0) - F_{h_n} (j_{h_n} p_0 + j_{h_n} \Phi_0  \psi_{0} p_{0})\\
&+ F_{h_n} (j_{h_n} p_0 + j_{h_n} \Phi_0  \psi_0 p_0) - j_{h_n} F_0(p_0 + \Phi_0  \psi_{0} p_{0}).
\end{align*}
By Lemma \ref{compare psi}, we have
\begin{align*}
	\| F_{h_n} (p_{n} + &\Phi_{h_n} \psi_{h_n} p_{n}) - j_{h_n} F_0 (p_0 + \Phi_0 \psi_0 p_0) \|_{L^2(\Omega_{h_n})} \\
&\leq 2L_{h_n}\|p_{n}(s) - j_{h_n}p_0(s)\|_{L^2(\Omega_{h_n})} + L_{h_n}\beta_{h_n} (\psi_0 B_{-T}) +  CL_{h_n}\alpha(h_n) + \rho(h_n).
\end{align*}
Hence we get
\begin{align} \label{I2}
\| II \|_{L^2(\Omega_{h_n})} 
&\leq  2L_{h_n} \int_t^0 e^{-A_{h_n} (t-s)} \|p_{n}(s) - j_{h_n} p_0(s)\|_{L^2(\Omega_{h_n})} ds + L_{h_n} \int_t^0 e^{-A_{h_n} (t-s)} \beta_{h_n} (\psi_0 B_{-T}) ds \notag\\
& \quad + CL_{h_n}\int_t^0e^{-A_{h_n} (t-s)} \alpha({h_n}) ds
+ \int_t^0 e^{-A_{h_n} (t-s)} \rho({h_n}) ds \notag \\
&\leq 2L_{h_n} \int_t^0 e^{-(\lambda_m^0 +1) (t-s)} \|p_{n}(s) - j_{h_n} p_0(s)\|_{L^2(\Omega_{h_n})} ds \notag\\
& \quad + L_{h_n} \int_t^0 e^{-(\lambda_m^0 +1) (t-s)} \beta_{h_n} (\psi_0 B_{-T}) ds
+  CL_{h_n}\int_t^0 e^{-(\lambda_m^0 +1) (t-s)} \alpha({h_n}) ds \notag\\
& \quad + \int_t^0 e^{-(\lambda_m^0 +1) (t-s)} \rho({h_n}) ds  \notag   \\
&\leq 2L_{h_n} e^{-(\lambda_m^0 +1)t} \int_t^0 e^{(\lambda_m^0+1)s} \|p_{n} (s) - j_{h_n} p_0(s) \|_{L^2(\Omega_{h_n})} ds \notag \\
& \quad + L_{h_n} \dfrac{e^{-(\lambda_m^0 +1)t}}{\lambda_m^0 +1}  \beta_{h_n} (\psi_0 B_{-T}) +  CL_{h_n}\dfrac{e^{-(\lambda_m^0 +1)t}}{\lambda_m^0 +1}  \alpha({h_n})
+ \dfrac{e^{-(\lambda_m^0 +1)t}}{\lambda_m^0 +1}  \rho({h_n}),
\end{align}
where we have used the fact that $\int_t^0 e^{(\lambda_m^0 +1)s} < \frac{1}{\lambda_m^0 +1}$ in the last inequality.

Step 3. We estimate $III$ for $t \in [-T,0]$. For this, we first consider the following
\begin{align*}
&(j_{h_n} e^{-A_0 (t-s)} P_m^0 - e^{-A_{h_n} (t-s)} P_m^{h_n} j_{h_n}) (v_0) \\
&= j_{h_n} \sum_{i=1}^m (e^{-\lambda_i^0 (t-s)} - e^{-\lambda_i^{h_n} (t-s)} ) b_i \phi_i^0 
+ \sum_{i=1}^m e^{-\lambda_i^{h_n} (t-s)} b_i (j_{h_n} \phi_i^0 - \phi_i^{h_n})\\
& \ \ \ + e^{-A_{h_n} (t-s)} (P_m^{h_n} v_{n} - P_m^{h_n} j_{h_n} v_0)\\
&:= III_1 + III_2 + III_3,
\end{align*}
where $v_0 = \sum_{i=1}^\infty b_i\phi_i^0 \in F_0(\hat{B}_{-T})$, and $v_{n} = \sum_{i=1}^\infty b_i \phi_i^{h_n} \in L^2(\Omega_{h_n})$.
For any $s \in (t,0]$, we have
\begin{align*}
\|III_1 \|_{L^2(\Omega_{h_n})} 
&= \left\| j_{h_n} \sum_{i=1}^m (e^{-\lambda_i^0 (t-s)} - e^{-\lambda_i^{h_n} (t-s)} ) b_i \phi_i^0 \right\|_{L^2(\Omega_{h_n})} \\
&\leq 2 \left\| \sum_{i=1}^m (e^{-\lambda_i^0 (t-s)} - e^{-\lambda_i^{h_n} (t-s)} ) b_i \phi_i^0 \right\|_{L^2(\Omega_{0})} \leq 2 \gamma_{h_n} (T)  \sum_{i=1}^m |b_i| \leq  2 C \gamma_{h_n} (T),
\end{align*}

\begin{align*}
\|III_2 \|_{L^2(\Omega_{h_n})} 
&= \left \| \sum_{i=1}^m e^{-\lambda_i^{h_n} (t-s)} b_i (j_{h_n} \phi_i^0 - \phi_i^{h_n}) \right \|_{L^2(\Omega_{h_n})}  \\ 
& \le e^{-(\lambda_m^0 +1 ) (t-s)} \alpha({h_n}) \sum_{i=1}^m |b_i| \leq C e^{-(\lambda_m^0 +1 ) (t-s)} \alpha({h_n} ), ~\mbox{and}\\
\|III_3 \|_{L^2(\Omega_{h_n})}^2 &= \sum_{i=1}^m e^{-2\lambda_i^{h_n} (t-s)} |(P_m^{h_n} v_{n} - P_m^{h_n} j_{h_n} v_0,\phi_i^{h_n})|^2 \leq e^{-2(\lambda_m^0 +1 ) (t-s)} \left(\alpha(h_n) \sum_{i=1}^m |b_i| \right)^2.
\end{align*}
Hence we see that
\begin{align*}
\|III_3 \|_{L^2(\Omega_{h_n})}
\leq e^{-(\lambda_m^0 +1 ) (t-s)} \alpha(h_n) \sum_{i=1}^m |b_i|
\leq C e^{-(\lambda_m^0 +1 ) (t-s)} \alpha({h_n}).
\end{align*}
Since $\int_t^0 e^{(\lambda_m^0 +1)s} ds <\frac{1}{\lambda_m^0 +1}$, we have
\begin{align} \label{I3}
\|III \|_{L^2(\Omega_{h_n})} &\leq \int_t^0 2 C \gamma_{h_n} (T) ds
+ \int_t^0 2 C e^{-(\lambda_m^0 +1 ) (t-s)} \alpha(h_n) ds \notag\\
&\leq 2T C \gamma_{h_n}(T) + 2 C \alpha({h_n})  \dfrac{e^{-(\lambda_m^0 +1)t}}{\lambda_m^0+1}.
\end{align}

Step 4. We estimate $\|p_{n}(t) - j_{h_n} p_0(t)\|_{L^2(\Omega_{h_n})}$ for $t \in [-T,0]$.
By putting \eqref{I1}, \eqref{I2} and \eqref{I3} together into \eqref{p-jp}, we get 
\begin{align}
\|p_{n}(t) - j_{h_n} p_0(t)\|_{L^2(\Omega_{h_n})}
\leq& ~ C \gamma_{h_n} (T) 
+ C e^{-(\lambda_m^0 +1)t} \alpha({h_n}) \notag \\
&+ 2e^{-(\lambda_m^0 +1) t}L_{h_n} \int_t^0 e^{(\lambda_m^0 +1) s}
\|p_{n}(s) - j_{h_n} p_0(s)\|_{L^2(\Omega_{h_n})} ds \notag \\
&+ \dfrac{e^{-(\lambda_m^0 +1) t}}{\lambda_m^0 +1}  L_{h_n} \beta_{h_n} (\psi_0 B_{-T}) + \rho({h_n}) \dfrac{e^{-(\lambda_m^0 +1) t}}{\lambda_m^0 +1} \notag  \\
&+ 2T C \gamma_{h_n} (T)
+ C(2 +L_{h_n}) \alpha(h_n) \dfrac{e^{-(\lambda_m^0 +1)t}}{\lambda_m^0 +1} .\label{estimate p-}
\end{align}
Let  $g(t)= e^{(\lambda_m^0 +1) t} \|p_{h_n}(t) -j_{h_n} p_0(t)\|_{L^2(\Omega_{h_n})}$. Multiply both sides of \eqref{estimate p-} by $e^{(\lambda_m^0 +1) t}$ to get
\begin{align*}
g(t) 
\leq&  e^{(\lambda_m^0 +1) t} C \gamma_{h_n} (T)
+ C \alpha({h_n})
+ 2L_{h_n} \int_t^0 g(s) ds  \\
&+ \dfrac{1}{\lambda_m^0 +1} L_{h_n} \beta_{h_n} (\psi_0 B_{-T})  + \dfrac{1}{\lambda_m^0 +1} \rho(h_n) 
+ 2T e^{(\lambda_m^0 +1) t} C \gamma_{h_n} (T) 
+ \dfrac{C(2 +L_{h_n})}{\lambda_m^0 +1} \alpha(h_n).
\end{align*}

\noindent
By applying the Gronwall's inequality, we derive that 
\begin{align*}
g(t)
\leq& \Big( e^{(\lambda_m^0 +1) t} C \gamma_{h_n} (T) 
+ C \alpha(h_n)
+ \dfrac{1}{\lambda_m^0 +1} L_{h_n} \beta_{h_n} (\psi_0 B_{-T})  + \dfrac{1}{\lambda_m^0 +1} \rho({h_n})\\
&+ 2T e^{(\lambda_m^0 +1) t} C \gamma_{h_n} (T)
+ \dfrac{ C(2 +L_{h_n})}{\lambda_m^0 +1} \alpha(h_n) \Big) e^{2L_{h_n} t}.
\end{align*}

\noindent
Consequently for any $t \in [-T,0]$, we have
\begin{align*}
\|p_{n}(t) - j_{h_n} p_0(t) \|_{L^2(\Omega_{h_n})}
\leq& \Bigg( e^{(\lambda_m^0 +1) t} C \gamma_{h_n} (T) 
+ C \alpha(h_n)
+ \dfrac{1}{\lambda_m^0 +1} L_{h_n} \beta_{h_n} (\psi_0 B_{-T})  + \dfrac{1}{\lambda_m^0 +1} \rho({h_n})\\
&+ 2T e^{(\lambda_m^0 +1) t} C\gamma_{h_n} (T)
+ \dfrac{ C(2 +L_{h_n})}{\lambda_m^0 +1} \alpha(h_n) \Bigg) e^{(2L_{h_n}-\lambda_m^0 -1) t}.
\end{align*}

Step 5. Finally we estimate $\|p_{n}(t) - j_{h_n} p_0(t)\|_{L^2(\Omega_{h_n})}$ for $t \in [0,T]$.
By the same techniques as in Step 1, we have
\begin{align*}
\|I\|_{L^2(\Omega_{h_n})} \leq C \gamma_{h_n}(T)  + e^{-rt} C \alpha({h_n}).
\end{align*}

\noindent
Furthermore we obtain
\begin{align*}
\|II\|_{L^2(\Omega_{h_n})} &\leq 2L_{h_n} \int_0^t e^{-A_{h_n} (t-s)} \|p_{n}(s) - j_{h_n} p_0(s) \|_{L^2(\Omega_{h_n})}  ds
+ L_{h_n} \beta_{h_n} (\psi_0 B_T) \int_0^t e^{-A_{h_n} (t-s)} ds \\
&\quad + CL_{h_n}\alpha(h_n)\int_0^t e^{-A_{h_n} (t-s)} ds
+ \int_0^t e^{-A_{h_n} (t-s)} \rho({h_n})ds\\
&\leq 2L_{h_n} e^{-rt} \int_0^t e^{rs} \|p_{n} (s) -j_{h_n} p_0(s) \|_{L^2(\Omega_{h_n})} ds 
+ L_{h_n} \beta_{h_n} (\psi_0 B_T) e^{-rt} \int_0^t e^{rs} ds \\
&\quad + CL_{h_n}\alpha(h_n) e^{-rt} \int_0^t e^{rs} ds 
+ \rho(h_n) e^{-rt} \int_0^t e^{rs}ds\\
&\leq 2L_{h_n} e^{-rt} \int_0^t e^{rs} \|p_{n}(s) -j_{h_n} p_0(s) \|_{L^2(\Omega_{h_n})} ds 
+ \dfrac{L_{h_n}}{r} \beta(h_n) + \dfrac{CL_{h_n}}{r} \alpha(h_n) 
+ \dfrac{\rho({h_n})}{r},
\end{align*}
where we have used the fact $\int_0^t e^{rs}ds \leq e^{rt}/ r$  for the last inequality.

For the estimate of $III$, we consider
\begin{align*}
\|III_1 \|_{L^2(\Omega_{h_n})} &\leq 2 \gamma_{h_n}(T)\sum_{i=1}^m |b_i| \leq 2 C \gamma_{h_n}(T), \\
\|III_2 \|_{L^2(\Omega_{h_n})}
&\leq \left\|\sum_{i=1}^m e^{-\lambda_i^h (t-s)} b_i(j_{h_n} \phi_i^0 - \phi_i^{h_n}) \right\|_{L^2(\Omega_{h_n})}
\\
&\leq e^{-r(t-s)} \alpha({h_n}) \sum_{i=1}^m |b_i|
\leq e^{-r(t-s)} C \alpha({h_n}), ~\mbox{and}\\
\|III_3 \|_{L^2(\Omega_{h_n})} & \leq e^{-r(t-s)} \alpha({h_n}) \sum_{i=1}^m |b_i| \leq  e^{-r(t-s)} C \alpha({h_n}).
\end{align*}
Then we get
\begin{align*}
\|III\|_{L^2(\Omega_{h_n})}  &\leq \int_0^t 2C \gamma_{h_n}(T) ds
+ \int_0^t 2 e^{-r (t-s)} C \alpha({h_n})  ds \leq 2T C \gamma_{h_n}(T) + \dfrac{2 C }{r} \alpha(h_n).
\end{align*}
Consequently we derive that
\begin{align}
\|p_{n}(t) -j_{h_n} p_0(t)\|_{L^2(\Omega_{h_n})}  
\leq& 
2 e^{-r t}L_{h_n} \int_0^t e^{r s} \|p_{n}(s) - j_{h_n} p_0(s)\|_{L^2(\Omega_{h_n})} ds + C \gamma_{h_n}(T)  + e^{-rt} C \alpha({h_n}) \notag \\
& + \dfrac{L_{h_n}}{r}  \beta_{h_n} (\psi_0 B_T)
+ \dfrac{\rho({h_n})}{r} + 2T C \gamma_{h_n}(T)
+ \dfrac{ C(2+L_{h_n})}{r} \alpha({h_n}). \label{estimate p+}
\end{align}
Let $g(t) = e^{rt} \|p_{n}(t) -j_{h_n} p_0(t)\|_{L^2(\Omega_{h_n})}$. Multiply both sides of \eqref{estimate p+} by $e^{rt}$ to deduce that
\begin{align*}
g(t)
&\leq 2 L_{h_n} \int_0^t g(s) ds + e^{r t} C \gamma_{h_n}(T)+ C \alpha(h_n) + e^{rt} \dfrac{L_{h_n}}{r} \beta_{h_n} (\psi_0 B_T) \\
& \quad + e^{rt} \dfrac{\rho({h_n})}{r} 
+ 2T e^{rt} C \gamma_{h_n}(T) 
+  e^{rt}\dfrac{ C(2+L_{h_n})}{r} \alpha({h_n}).
\end{align*}
By the Gronwall's inequality, we get

\begin{align*}
g(t)
\leq& \Bigg( e^{r t} C \gamma_{h_n}(T)
+ C \alpha(h_n) 
+ e^{rt} \dfrac{L_{h_n}}{r} \beta_{h_n} (\psi_0 B_T) + e^{rt} \dfrac{\rho(h_n)}{r} \\
&+ 2T e^{r t} C \gamma_{h_n}(T) 
+ e^{rt}\dfrac{ C(2+L_{h_n})}{r} \alpha({h_n}) \Bigg) e^{2L_{h_n} t}.
\end{align*}

Finally we deduce that
\begin{align*}
\|p_{n}(t) -j_{h_n} p_0(t) \|_{L^2(\Omega_{h_n})}
&\leq   \Bigg( e^{r t} C \gamma_{h_n}(T)
+ C \alpha(h_n) 
+ e^{rt} \dfrac{L_{h_n}}{r} \beta_{h_n} (\psi_0 B_T) + e^{rt} \dfrac{\rho(h_n)}{r}\\
& \quad + 2T e^{r t} C\gamma_{h_n}(T) 
+ e^{rt}\dfrac{ C(2+L_{h_n})}{r} \alpha({h_n}) \Bigg) e^{(2L_{h_n}-r) t}.
\end{align*}
\end{proof}

In the following lemma, we estimate the linear semigroups of orthogonal complements.

\begin{lemma} \label{orthogonal_est}

For any $\varepsilon >0$, $T>0$ and a bounded subset $B$ of $L^2(\Omega_0)$, there is $K>0$ such that for any $u \in B$ and $n \geq K$
$$
\int_0^T \|e^{-A_{h_n}t} Q_m^{h_n} j_{h_n} u - j_{h_n} e^{-A_0t} Q_m^0 u \|_{L^2(\Omega_{h_n})} dt < \varepsilon.
$$
\end{lemma}

\begin{proof}
Since $B$ is bounded,  we can choose $\delta > 0$ and $k \in \mathbb{N}$ $(k > m)$ such that 
$$
4\delta \|u\|_{L^2(\Omega_0)} <{\varepsilon}/2~\mbox{and}~2 e^{-(\lambda_{k+1}^0 -1)\delta} \|u \|_{L^2(\Omega_0)}< \varepsilon/{6(T-\delta}), ~\forall u \in B.
$$
By Proposition \ref{continuity of spectra} and Lemma \ref{alpha}, we can take a subsequence of $\{h_n\}_{n \in \mathbb{N}}$, still denoted by $\{h_n\}_{n \in \mathbb{N}}$, and the first $k$ eigenfunctions, denoted by $\{\phi_1^0, \ldots, \phi_k^0\}$, with respect to $k$ eigenvalues $\{\lambda_1^0, \ldots, \lambda_k^0\}$ such that
\begin{align*}
\gamma_k (h_n)&:=\sup \{ |e^{-\lambda_i^{h_n} t} - e^{-\lambda_i^0 t}| : 1 \leq i \leq k, 0 \leq t \leq T  \} \rightarrow 0, ~\mbox{and} \\
\alpha_k (h_n)&:= \sup \{\|\phi_i^{h_n} - j_{h_n} \phi_i^0 \|_{L^2(\Omega_{h_n})} : 1 \leq i \leq k \} \rightarrow 0 ~\mbox{as}~ n \rightarrow \infty.
\end{align*}
For any $u = \sum_{i=1}^\infty a_i \phi_i^0$  in $B$ and $t \in [0,\delta]$, we have 
\begin{align*}
&\|e^{-A_{h_n}t} Q_m^{h_n} j_{h_n} u- j_{h_n} e^{-A_0t} Q_m^0 u \|_{L^2(\Omega_{h_n})} \notag \\
&\le \|e^{-A_{h_n}t} Q_m^{h_n} j_{h_n} u\|_{L^2(\Omega_{h_n})}+ 2 \|e^{-A_0t} Q_m^0 u \|_{L^2(\Omega_0)} \le 4\|u\|_{L^2(\Omega_0)}.
\end{align*}
This implies that
\begin{equation}  \label{delta}
\int_0^{\delta}\|e^{-A_{h_n}t} Q_m^{h_n} j_{h_n} u- j_{h_n} e^{-A_0t} Q_m^0 u \|_{L^2(\Omega_{h_n})} dt < \dfrac{\varepsilon}{2}.
\end{equation}

For any $t \in [\delta, T]$, we obtain
\begin{align*}
\|e^{-A_{h_n}t}& Q_m^{h_n} j_{h_n} u- j_{h_n} e^{-A_0t} Q_m^0 u\|_{L^2(\Omega_{h_n})}   \\
& \leq \left\| \sum_{i=m+1}^k e^{-\lambda_i^{h_n} t} a_i Q_m^{h_n} j_{h_n} \phi_i^0 - j_{h_n} \sum_{i=m+1}^k e^{-\lambda_i t} a_i \phi_i^0 \right\|_{L^2(\Omega_{h_n})}   \\
& \ \ \ + \left\| \sum_{i=k+1}^\infty  e^{-\lambda_i^{h_n} t} a_i Q_m^{h_n} j_{h_n} \phi_i^0 \right\|_{L^2(\Omega_{h_n})} + \left\| j_{h_n} \sum_{i=k+1}^\infty e^{-\lambda_i^0 t}  a_i \phi_i^0  \right\|_{L^2(\Omega_{h_n})} \\
&:= I + II + III.
\end{align*}

We first estimate $I$ as follows.
\begin{align*}
I &\leq \left\| \sum_{i=m+1}^k  ( e^{-\lambda_i^{h_n} t} -  e^{-\lambda_i^0 t} ) a_i Q_m^{h_n} j_{h_n} \phi_i^0\right\|_{L^2(\Omega_{h_n})} \\
& \ \ \ + \left\| \sum_{i=m+1}^k  e^{-\lambda_i^0 t} a_i Q_m^{h_n} j_{h_n} \phi_i^0 
-  e^{-\lambda_i^0 t} a_i \phi_i^{h_n}  \right\|_{L^2(\Omega_{h_n})} \\
& \ \ \ + \left\| \sum_{i=m+1}^k e^{-\lambda_i^0 t} a_i \phi_i^{h_n}
-  \sum_{i=m+1}^k e^{-\lambda_i^0 t} a_i j_{h_n} \phi_i^0  \right\|_{L^2(\Omega_{h_n})}   \\
&\leq \gamma_k(h_n) \left\| \sum_{i=m+1}^k a_i Q_m^{h_n} j_{h_n} \phi_i^0 \right\|_{L^2(\Omega_{h_n})}
+2 \alpha_k(h_n) \sum_{i=m+1}^k  |a_i|.
\end{align*}
Since $\gamma_k(h_n) \rightarrow 0$ and $\alpha_k(h_n) \rightarrow 0$ as $n \rightarrow \infty$, there exists $K \in \mathbb{N}$ such that if $n \ge K$, then we have $I < \varepsilon /6(T-\delta)$. 

On the other hand, by the choice of $\delta$ and $k$, we have
\begin{align*}
II &\leq \|j_{h_n}\|  \sum_{i=k+1}^\infty e^{-\lambda_i^{h_n} t} |a_i|
\leq 2 e^{-\lambda_{k+1}^h \delta} \|u \|_{L^2(\Omega_0)} < \dfrac{\varepsilon}{{6(T-\delta})}, ~\mbox{and}~\\
III &\leq \|j_{h_n} \| \sum_{i=k+1}^\infty e^{-\lambda_i^{h_n} t} |a_i|
\leq 2 e^{-\lambda_{k+1}^h \delta} \|u \|_{L^2(\Omega_0)} < \dfrac{\varepsilon}{{6(T-\delta})}.
\end{align*}

Consequently we get
\begin{equation}\label{delta comp}
\int_{\delta}^T\|e^{-A_{h_n}t} Q_m^{h_k} j_{h_k} u - j_{h_k} e^{-A_0 t} Q_m^0 u \|_{L^2(\Omega_{h_k})} dt  \le \int_{\delta}^T (I + II + III) dt < \dfrac{\varepsilon}{2}, ~\forall u \in B.
\end{equation}
By \eqref{delta} and \eqref{delta comp}, we derive that
$$
\int_0^T\|e^{-A_{h_k}t} Q_m^{h_k} j_{h_k} u - j_{h_k} e^{-A_0 t} Q_m^0 u \|_{L^2(\Omega_{h_k})} dt<{\varepsilon}.
$$
This completes the proof.
\end{proof}

In the proof of Theorem \ref{existence of IM}, we know that $\lambda_{m+1}^{h} - \lambda_{m}^h > 2 \sqrt{2} L_h$ if $d_{C^1}(h,id)$ is sufficiently small. Then by applying the Lyapunov-Perron method, we see that 
\begin{align}
\Phi_{h_n}(p)  &= \int_{-\infty}^0 e^{A_{h_n} s} Q_m^{h_n} F_{h_n}(p_n(s) + \Phi_{h_n}(\psi_{h_n}p_n(s)))ds,  ~\mbox{and} \notag \\
\Phi_0 (p)  &= \int_{-\infty}^0 e^{A_{0} s} Q_m^{0} F_{0}(p_0(s) + \Phi_{0}(\psi_{0}p_0(s)))ds, ~\forall  p \in \mathbb R^m,~\forall  n \in \mathbb N, \label{lyapunov}
\end{align}
where $p_n(t)$ and $p_0(t)$ are the solutions of  \eqref{eq:p_h} and \eqref{eq:p_0} with initial conditions $p_n(0) = \psi_{h_n}^{-1}(p)$ and $p_0(0) = \psi_0^{-1}(p)$, respectively, for some $p \in \mathbb{R}^m$ (for more details, see \cite{Ro}).
Since $\overline{d}_{C^1} (f_{h_n},f_0) \rightarrow 0$ as $n \rightarrow \infty$, we can take $M_F >0$ such that for sufficiently large $n$, 
$$
\max\{\|F_0 (u_0)\|_{L^2(\Omega_0)},~\|F_{h_n} (u_n)\|_{L^2(\Omega_{h_n})} \} \le M_F, ~\forall u_0 \in L^2(\Omega_0), ~\forall u_n \in L^2(\Omega_{h_n}).
$$
By Theorem 1 in \cite{Ro}, we see that
$$
\|\Phi_0 (p)\|_{L^2(\Omega_0)}< \dfrac{M_F }{ \lambda_{m+1}^0} ~\mbox{and}~ \|\Phi_{h_n} (p)\|_{L^2(\Omega_{h_n})}< \dfrac{M_F }{ \lambda_{m+1}^{h_n}}.
$$
By Proposition \ref{continuity of spectra}, we can assume that $\lambda_{m+1}^{h_n} \rightarrow \lambda_{m+1}^0$ as $n \rightarrow \infty$. Then there is $M>0$ such that 
\begin{equation} \label{beta}
\beta_{h_n} (\mathbb{R}^m) < M, ~\forall n \in \mathbb{N}.
\end{equation}
For simplicity, we denote $F_{h_n}$ and $F_0$ by $ F_{h_n}=F_{h_n}(p_n + \Phi_{h_n}(\psi_{h_n}p_n))$ and $ F_0=F_0(p_0 + \Phi_{0}(\psi_{0}p_0))$. With the notations, we have the following lemma.

\begin{lemma} \label{distance of IM}
For any bounded set $B \subset \mathcal{M}_0$, 
$\beta_{h_n} (\psi_0 P_m^0 B) \rightarrow 0$ as $n \rightarrow \infty$.
\end{lemma}

\begin{proof}
Let $\varepsilon > 0$ be arbitrary, and choose $\delta >0$ such that
$$
\eta :=  \dfrac{2}{2\sqrt{2} +1 - \delta} + \dfrac{1}{2 \sqrt{2} + 1 -\delta} <1,
$$
and denote by $\eta_0 = \sum_{i=1}^\infty \eta^i$.
Take a constant $T>0$ such that
$$
\int_{-\infty}^{-T} \| e^{A_{h_n} s} Q_m^{h_n} F_{h_n} - j_{h_n} e^{A_0 s} Q_m^0 F_0\|_{L^2(\Omega_{h_n})} ds  \le \dfrac{\varepsilon}{8 \eta_0}.
$$
For any  $k \geq 1$ and a bounded set $B \subset \mathcal{M}_0$, we denote by 
$$
\hat B_{-kT} = S_0 (B,[-kT,0]) ~\mbox{and}~ B_{-kT} = P_m^0 S_0 (B,[-kT,0]).
$$
Since $F_0 (\hat B_{-T})$ is bounded in $L^2(\Omega_0)$, there exists a constant $C>0$ such that for any $u = \sum_{i=1}^m a_i \phi_i^0$ in $B_{-T}$ and $v = \sum_{i=1}^m b_i \phi_i^0$ in $P_m^0 F_0 (\hat B_{-T})$, we have $\sum_{i=1}^m |a_i| < C$ and $\sum_{i=1}^m |b_i| <C$.

Step 1. There is $N_1 > 0$ such that for any $n \geq N_1$,
$$
\beta_{h_n}(\psi_0 P_m^0 B) \leq \eta \ \beta_{h_n} (\psi_0 B_{-T}) +  \dfrac{\varepsilon}{4 \eta_0}.
$$

For any $p \in B$, we have
\begin{align*}
\| \Phi_{h_n} (\psi_0 p) &- j_{h_n} \Phi_0 (\psi_0 p) \|_{L^2(\Omega_{h_n})}  \\
&\leq \int_{-\infty}^0 \| e^{A_{h_n} s} Q_m^{h_n} F_{h_n} - j_{h_n} e^{A_0 s} Q_m^0 F_0\|_{L^2(\Omega_{h_n})} ds   \\
&=  \int_{-\infty}^{-T} \| e^{A_{h_n} s} Q_m^{h_n} F_{h_n} - j_{h_n} e^{A_0 s} Q_m^0 F_0\|_{L^2(\Omega_{h_n})} ds   \\
& \quad + \int_{-T}^{0} \| e^{A_{h_n} s} Q_m^{h_n} F_{h_n} - j_{h_n} e^{A_0 s} Q_m^0 F_0 \|_{L^2(\Omega_{h_n})} ds \\
&\leq \dfrac{\varepsilon}{8 \eta_0}
+ \int_{-T}^0 \| e^{A_{h_n} s} Q_m^{h_n} (F_{h_n} -j_{h_n} F_0) \|_{L^2(\Omega_{h_n})}ds \\
& \quad + \int_{-T}^0 \| (e^{A_{h_n} s} Q_m^{h_n} j_{h_n} - j_{h_n} e^{A_0 s} Q_m^0)F_0 \|_{L^2(\Omega_{h_n})} ds : = \dfrac{\varepsilon}{8 \eta_0} + I + II.
\end{align*}
By Lemma \ref{compare p(t)}, we obtain
\begin{align*}
\|e^{A_{h_n} s}& Q_m^{h_n} (F_{h_n} - j_{h_n} F_0)\|_{L^2(\Omega_{h_n})}  
\leq e^{\lambda_{m+1}^{h_n} s} \|(F_{h_n} - j_{h_n} F_0)\|_{L^2(\Omega_{h_n})} \\
&\leq e^{\lambda_{m+1}^{h_n} s} \left(2L_{h_n} \|p_n (s) - j_{h_n} p_0(s)\|_{L^2(\Omega_{h_n})} + L_{h_n} \beta_{h_n} (\psi_0 B_{-T}) + CL_{h_n}\alpha(h_n)+ \rho(h_n) \right)  \\
&\leq \Bigg( 2e^{(\lambda_m^0 +1) t} L_{h_n} C \gamma_{h_n}(T) 
+2 L_{h_n} C \alpha(h_n)
+ \dfrac{2}{\lambda_m^0 +1} L_{h_n}^2 \beta_{h_n} (\psi_0 B_{-T}) 
+ \dfrac{2}{\lambda_m^0 +1} L_{h_n} \rho({h_n})\\
& \quad + 4T e^{(\lambda_m^0 +1) t} L_{h_n} C\gamma_{h_n} (T)
+ \dfrac{ 2C(2 +L_{h_n})}{\lambda_m^0 +1} L_{h_n} \alpha(h_n) \Bigg) e^{(2L_{h_n} + \lambda_{m+1}^{h_n} - \lambda_m^0 -1)s} \\
& \quad + e^{\lambda_{m+1}^{h_n} s}L_{h_n} \beta_{h_n} (\psi_0 B_{-T}) 
+ e^{\lambda_{m+1}^{h_n} s}CL_{h_n}\alpha(h_n)
+ e^{\lambda_{m+1}^{h_n} s} \rho(h_n).
\end{align*}
Hence we get
\begin{align}
I =&\int_{-T}^0 \|e^{A_{h_n} s} Q_m^{h_n} (F_{h_n} - j_{h_n} F_0)\|_{L^2(\Omega_{h_n})}
\leq \dfrac{2L_{h_n} C \gamma_{h_n}(T)}{2L_{h_n} + \lambda_{m+1}^{h_n}} \notag\\
& \quad + \dfrac{2L_{h_n} C \alpha(h_n)}{2L_{h_n} + \lambda_{m+1}^{h_n} - \lambda_m^0 -1}
+ \dfrac{2L_{h_n}^2 \beta_{h_n} (\psi_0 B_{-T})}{(\lambda_m^0 +1) (2L_{h_n} + \lambda_{m+1}^{h_n} - \lambda_m^0 -1)}  \notag\\
& \quad + \dfrac{2L_{h_n} \rho(h_n)}{2L_{h_n} + \lambda_{m+1}^{h_n} - \lambda_m^0 -1}
+ \dfrac{4L_{h_n} T C\gamma_{h_n}(T)}{2L_{h_n} + \lambda_{m+1}^{h_n}}\notag \\ 
& \quad + \dfrac{2 C (2+L_{h_n})L_{h_n}\alpha(h_n)}{(\lambda_m^0 +1) (2L_{h_n} + \lambda_{m+1}^{h_n} - \lambda_m^0 -1)}
+ \dfrac{L_{h_n} \beta_{h_n} (\psi_0 B_{-T})}{\lambda_{m+1}^{h_n}}
+\dfrac{CL_{h_n}\alpha(h_n)}{\lambda_{m+1}^{h_n}}
+ \dfrac{\rho(h_n)}{\lambda_{m+1}^{h_n}} \notag\\
&\leq \left(\dfrac{2L_{h_n}^2 }{(\lambda_m^0 +1) (2L_{h_n} + \lambda_{m+1}^{h_n} - \lambda_m^0 -1)} 
+ \dfrac{L_{h_n}}{\lambda_{m+1}^{h_n}}  \right) \beta_{h_n} (\psi_0 B_{-T}) \notag \\
& \quad + \dfrac{2 C L_{h_n} \gamma_{h_n}(T)}{2L_{h_n} + \lambda_{m+1}^{h_n}}
+ \dfrac{2L_{h_n} C \alpha(h_n)}{2L_{h_n} + \lambda_{m+1}^{h_n} - \lambda_m^0-1} \notag \\
& \quad + \dfrac{2L_{h_n} \rho(h_n)}{2L_{h_n} + \lambda_{m+1}^{h_n} - \lambda_m^0 -1}
+ \dfrac{4L_{h_n} C T \gamma_{h_n}(T)}{2L_{h_n} + \lambda_{m+1}^{h_n}}\notag \\
& \quad + \dfrac{2 C(2+L_{h_n})L_{h_n} \alpha(h_n)}{(\lambda_m^0 +1) (2L_{h_n} + \lambda_{m+1}^{h_n} - \lambda_m^0 -1)} +\dfrac{CL_{h_n}\alpha(h_n)}{\lambda_{m+1}^{h_n}}
+ \dfrac{\rho(h_n)}{\lambda_{m+1}^{h_n}} \notag \\
&:=  \left(\dfrac{2L_{h_n}^2 }{(\lambda_m^0 +1) (2L_{h_n} + \lambda_{m+1}^{h_n} - \lambda_m^0 -1)} + \dfrac{L_{h_n}}{\lambda_{m+1}^{h_n}}  \right) \beta_{h_n} (\psi_0 B_{-T})  + \tilde I. \label{I_est}
\end{align}

\noindent
By Proposition \ref{continuity of spectra}, we can take $N_1 >0$ such that for any $n \geq N_1$,
$$
\lambda_m^{h_n} > \lambda_m^0 - \delta \quad \hbox{ and } \quad \lambda_m^{h_n} > L_{h_n} - \delta.
$$
Note that $\lambda_m^0 +1> L_{h_n}$ and
\begin{align*}
2L_{h_n} + \lambda_{m+1}^{h_n} - \lambda_m^0 -1
= 2L_{h_n} + (\lambda_{m+1}^{h_n} - \lambda_m^{h_n}) + (\lambda_m^{h_n} - \lambda_m^0) -1 > 2 L_{h_n} + 2 \sqrt{2} L_{h_n} -1 - \delta.
\end{align*}
Thus we have
\begin{align*}
\dfrac{2L_{h_n}^2}{(\lambda_m^0 +1) (2L_{h_n} + \lambda_{m+1}^{h_n} - \lambda_m^0 -1)} \leq \dfrac{2}{2\sqrt{2} + 1 - \delta}  \quad \hbox{ and } \quad \dfrac{L_{h_n}}{\lambda_{m+1}^{h_n}} <\dfrac{1}{2 \sqrt{2} + 1 -\delta}.
\end{align*}

Since $\gamma_{h_n}(T)$, $\alpha(h_n)$ and $\rho(h_n)$ converge to $0$ as $n \rightarrow \infty$, we can choose $N_1>0$ such that $\tilde I < \dfrac{\varepsilon}{16 \eta_0}$ for any $n \geq N_1$. Consequently we obtain
\begin{equation} \label{I}
I < \eta~ \beta_{h_n} (\psi_0 B_{-T}) + \dfrac{\varepsilon}{16 \eta_0},~ \forall n \geq N_1.
\end{equation}

On the other hand, by Lemma \ref{orthogonal_est}, we can take $N>0$ such that for any $u \in L^2(\Omega_0)$ with $\|u\|_{L^2(\Omega_0)} \le C$ and $n \ge N$,
\begin{equation}
II = \int_{-T}^0 \| (e^{A_{h_n} s} Q_m^{h_n} j_{h_n} - j_{h_n} e^{A_0 s} Q_m^0) F_0 \|_{L^2(\Omega_{h_n})} ds < \dfrac{\varepsilon}{16 \eta_0}. \label{II}
\end{equation}
By \eqref{I} and \eqref{II}, we have
\begin{align*}
\| \Phi_{h_n} ({p}) - j_{h_n} \Phi_0 ({p}) \|_{L^2(\Omega_{h_n})} < \eta ~ \beta_{h_n} (\psi_0 B_{-T}) + \dfrac{\varepsilon}{4 \eta_0}.
\end{align*}
Since ${p}$ is arbitrary in $B$, we get 
$$
\beta_{h_n} (\psi_0 P_m^0 B) < \eta~\beta_{h_n}(\psi_0 B_{-T}) + \dfrac{\varepsilon}{4 \eta_0}, ~\forall n \geq N_1.
$$
This completes the proof of  Step 1.

Step 2. There is $N >0$ such that $\beta_{h_n} (\psi_0 B) < \varepsilon$ for any $n \geq N$.

By the same procedure as in Step 1, we derive that for each $k \in \mathbb{N}$, there is $N_k > N_{k-1}$ such that for any $n \geq N_k$,
$$
\beta_{h_n}(\psi_0 B_{-(k-1)T}) < \eta~\beta_{h_n}(\psi_0 B_{-kT}) + \dfrac{\varepsilon}{4 \eta_0}.
$$
Hence we have
\begin{align*}
\beta_{h_n}(\psi_0 P_m^0 B) 
< \eta^k \beta_{h_n} (\psi_0 B_{-kT}) + \dfrac{\varepsilon}{4 \eta_0} \sum_{i=0}^{k-1} \eta^i  < \eta^k M + \dfrac{\varepsilon}{4}.
\end{align*}
Take $k >0$ such that $\eta^k M < \varepsilon /2$ and $N > N_k$. Then for any $n >N$, we have $\beta_{h_n} (\psi_0 P_m^0 B) < \varepsilon$. This completes the proof.
\end{proof}

For each $n \in \mathbb{N}$, we define $\hat{j}_{h_n} : \mathcal{M}_0 \rightarrow \mathcal{M}_{h_n}$ by
$$
\hat{j}_{h_n} (p_0 + \Phi_0 (\psi_0 p_0)) = \psi_{h_n}^{-1} \psi_0 p_0 + \Phi_{h_n}(\psi_0 p_0),~\forall p_0 \in P_m^0 \mathcal M_0.
$$
It  is clear that $\hat j_{h_n}$ is a bijection with the inverse $\hat i_{h_n}$ given by
$$
\hat{i}_{h_n} (p_n + \Phi_{h_n} (\psi_{h_n} p_{n})) = \psi_{0}^{-1} \psi_{h_n} p_n + \Phi_{0} (\psi_{h_n} p_n),~\forall p_n \in P_m^{h_n} \mathcal M_{h_n}.
$$

\begin{lemma} \label{compare j}
Let $B$ be a bounded subset of $\mathcal{M}_0$. Then
$$
\| j_{h_n}(u) - \hat{j}_{h_n}(u)\|_{L^2(\Omega_{h_n})} \to 0 ~\hbox{ as } n \rightarrow \infty
$$ 
uniformly for $u \in B$.
\end{lemma}

\begin{proof}
Let $B$ be a bounded subset of $\mathcal{M}_0$.
For any $u \in B$, there exists $a_i \in \mathbb R$ $(1 \le i \le m)$ such that 
$$
u = p_0 + \Phi_0 (\psi_0 p_0) \quad \mbox{with}\quad p_0 = \sum_{i=1}^m a_i\phi_i^0 \in P_m^0 \mathcal M_0.
$$
By the fact that  $\|\psi_{h_n}^{-1}\|_{L^{\infty}(\mathbb R^m, P_m^0L^2(\Omega_{h_n}))}=1$ and Lemma \ref{compare psi}, we have
\begin{align*}
\|j_{h_n} (u) - \hat{j}_{h_n} (u)\|_{L^2(\Omega_{h_n})}
&= \| j_{h_n} (p_0 + \Phi_0 (\psi_0 p_0)) - \hat{j}_{h_n} (p_0 + \Phi_0 (\psi_0 p_0)) \|_{L^2(\Omega_{h_n})} \\
&\leq \|j_{h_n} p_0 - \psi_{h_n}^{-1} \psi_0 p_0 \|_{L^2(\Omega_{h_n})}
+ \|j_{h_n} \Phi_0 (\psi_0 p_0) - \Phi_{h_n} (\psi_0 p_0) \|_{L^2(\Omega_{h_n})} \\
&\leq \|\psi_{h_n}^{-1} \| | \psi_{h_n} j_{h_n} p_0 - \psi_0 p_0 |_{\mathbb{R}^m}
+ \beta_{h_n}(\psi_0 P_m^0 B)\\
&\leq \alpha(h_n)  \sum_{i=1}^m |a_i| + \beta_{h_n}(\psi_0 P_m^0B).
\end{align*}
By Lemmas \ref{alpha} and \ref{distance of IM}, we see that $\alpha(h_n)$ and $\beta(h_n)$ converge to $0$ as $n \rightarrow \infty$. Hence we derive that  $\|j_{h_n} (u) - \hat{j}_{h_n} (u)\|_{L^2(\Omega_{h_n})} \rightarrow 0$ as $n \rightarrow \infty$.
\end{proof}

\begin{proof}[\bf End of Proof of Theorem \ref{GH convergence}]
We first show that there is $N >0$ such that
$\hat{j}_{h_n}$ is an $\varepsilon$-isometry for all $n \geq N$. Since $B_0$ is bounded in $L^2(\Omega_0)$, we take $C>0$ such that $\|u_0\|_{L^2(\Omega_0)}< C$ for all $u_0 \in B_0$.

For the bounded set $B_0 \subset \mathcal{M}_0$, by Lemma \ref{compare j}, we can take $N >0$ such that if $n \geq N$ then
\begin{align*}
\left| \|j_{h_n} \| -1 \right| < \dfrac{\varepsilon}{6C}, ~\mbox{and}~ \|j_{h_n}(u) - \hat{j}_{h_n} (u) \|_{L^2(\Omega_{h_n})} < \dfrac{\varepsilon}{3}, ~ \forall u \in B_0.
\end{align*} 
For any $u, \tilde u \in B_0$, we let
$$
u = p_0 + \Phi_0 (\psi_0 p_0) ~\mbox{and}~\tilde{u} = \tilde{p}_0 + \Phi_0 (\psi_0 \tilde{p}_0) ~\mbox{for~some}~ p_0, \tilde p_0 \in P_m^0\mathcal M_0.
$$
For any $n \ge N$, we have
\begin{align*}
\| \hat{j}_{h_n} u_0 &- \hat{j}_{h_n} \tilde{u}_0 \|_{L^2(\Omega_{h_n})} - \| u_0 - \tilde{u}_0 \|_{L^2(\Omega_0)}\\
& \leq  \|\hat{j}_{h_n} (p_0 + \Phi_0 (\psi_0 p_0)) - j_{h_n} (p_0 + \Phi_0 (\psi_0 p_0)) \|_{L^2(\Omega_{h_n})}  \\
& \ \ \ + \|j_{h_n} (p_0 + \Phi_0 (\psi_0 p_0)) - j_{h_n} (\tilde{p}_0 + \Phi_0 (\psi_0 \tilde{p}_0)) \|_{L^2(\Omega_{h_n})}\\
& \ \ \ + \|j_{h_n} (\tilde{p}_0 + \Phi_0 (\psi_0 \tilde{p}_0)) - \hat{j}_{h_n} (\tilde{p}_0 + \Phi_0 (\psi_0 \tilde{p}_0)) \|_{L^2(\Omega_{h_n})}
 -\| u_0 - \tilde{u}_0 \|_{L^2(\Omega_0)}\\
& \leq \dfrac{2 \varepsilon}{3} + (\|j_{h_n}\| -1) \|u_0 - \tilde{u}_0\|_{L^2(\Omega_0)} < \varepsilon.
\end{align*}
Similarly we can show that $ \| u_0 - \tilde{u}_0 \|_{L^2(\Omega_0)}- \| \hat{j}_{h_n} u_0 - \hat{j}_{h_n} \tilde{u}_0 \|_{L^2(\Omega_{h_n})} <\varepsilon$.
This shows that $\hat j_{h_n}$ is an $\varepsilon$-isometry on $B_0$.

On the other hand, for any $u, \tilde{u} \in B_{h_n}$, let us take $v, \tilde{v} \in B_0$ such that $u= \hat{j}_{h_n} (v)$ and $\tilde{u} =\hat{j}_{h_n} (\tilde{v})$. Then we have
\begin{align*}
\big| \| \hat{i}_{h_n} (u) - \hat{i}_{h_n} (\tilde{u})\|_{L^2(\Omega_0)} &- \|u - \tilde{u} \|_{L^2(\Omega_{h_n})} \big| \\
&= \big| \| \hat{i}_{h_n} (\hat{j}_{h_n} (v)) - \hat{i}_{h_n} (\hat{j}_{h_n} (\tilde{v}))\|_{L^2(\Omega_0)}  - \| \hat{j}_{h_n} (v) - \hat{j}_{h_n} (\tilde{v}) \|_{L^2(\Omega_{h_n})}   \big|  \\
&= \big| \| v - \tilde{v} \|_{L^2(\Omega_0)}  - \| \hat{j}_{h_n} (v) - \hat{j}_{h_n} (\tilde{v}) \|_{L^2(\Omega_{h_n})}   \big| < \varepsilon.
\end{align*}

\noindent
This shows that $\hat{i}_{h_n}$ is an $\varepsilon$-isometry on $B_{h_n}$.

Moreover, since $\hat{j}_{h_n} (B_0)=B_{h_n}$ and $\hat{i}_{h_n}(B_{h_n})=B_0$ for all $n \in \mathbb{N}$, we get $d_{GH}(B_{h_n}, B_0) < \varepsilon$ for all $n \geq N$.
The contradiction completes the proof.
\end{proof}

\section{Proof of Theorem  \ref{stability of flows}}
Suppose not. Then there are $\varepsilon >0$, $T>0$, and a bounded set $B_0 \subset \mathcal{M}_0$
such that for any $n \in \mathbb{N}$, 
there is $h_n \in {\rm Diff}(\Omega_0)$ with $d_{C^1}(h_n,id)< 1/n$
such that for any bounded set $B_{h_n} \subset \mathcal{M}_{h_n}$,
$D_{GH}^T(S_{h_n}|_{B_{h_n}},S_0|_{B_0}) \geq \varepsilon$.

Let $\{\lambda_1^{h_n}, \ldots, \lambda_m^{h_n} \}$ and $\{\phi_1^{h_n}, \ldots, \phi_m^{h_n}\}$ be the first $m$ eigenvalues and corresponding eigenfunctions of $A_{h_n}$, respectively.
By Proposition \ref{continuity of spectra}, there are eigenfunctions $\{\phi_1^0, \ldots, \phi_m^0\}$ with respect to the first $m$ eigenvalues $\{\lambda_1^0, \ldots, \lambda_m^0\}$ of $A_0$, and a subsequence of $\{h_n\}_{n \in \mathbb N}$, still denoted by $\{h_n\}_{n \in \mathbb N}$, such that 
$$
\phi_i^{h_n} \rightarrow \phi_i^0 ~\mbox{in}~L^2(\mathbb R^N)~\mbox{as}~ n \rightarrow \infty, ~\forall 1 \leq i \leq m.
$$
For each $n \in \mathbb{N}$, we denote by $B_{h_n}$ the collection of $u_{h_n} \in \mathcal{M}_{h_n}$ such that $\psi_{h_n} P_m^{h_n} u_{h_n} = \psi_0 P_m^0 u_0$ for some $u_0 \in B_0$.

Now we show that $D_{GH}^T(S_{h_n}|_{B_{h_n}}, S_0|_{B_0})< \varepsilon$ for sufficiently large $n$.
Let $p_0(t)$ and $p_n(t)$ be the solutions of  \eqref{eq:p_0} and \eqref{eq:p_h}, respectively, such that $p_0(0) \in P_m^0B_0$ and $\psi_0p_0(0) = \psi_{h_n}p_n(0)$. By Lemma \ref{compare p(t)}, there is $N>0$ such that 
$$
\|p_n(t)\|_{L^2(\Omega_{h_n})} \leq \|j_{h_n} p_0(t)\|_{L^2(\Omega_{h_n})} +M, ~\forall t \in [-T,T],~\forall n \geq N,
$$
where $M>0$ is given in \eqref{beta}.
It follows that
\begin{align*}
\|\hat{i}_{h_n} (p_n(t) + \Phi_{h_n} (\psi_{h_n} p_n(t)) \|_{L^2(\Omega_0)}
&= \|\psi_0^{-1} \psi_{h_n} p_n(t) + \Phi_0(\psi_{h_n} p_n(t)) \|_{L^2(\Omega_0)} \\
&\leq \|p_n(t)\|_{L^2(\Omega_{h_n})} + M\\
&\leq \|j_{h_n} p_0(t)\|_{L^2(\Omega_{h_n})} + 2M \\
&\leq 2C + 2M,
\end{align*}
where $C = \sup \{\|S_0(u_0,t)\|_{L^2(\Omega_0)}: u_0 \in B_0, t \in [-T,T]\}$.
Then there is a bounded set $D_0 \subset \mathcal{M}_0$ such that
$$
S_0 (B_0,[-T,T]) \subset D_0 ~\hbox{and}~ \hat i_{h_n}(S_{h_n} (B_{h_n},[-T,T])) \subset D_0, \forall n \geq N.
$$
For each $n \in \mathbb{N}$, we denote by $D_{h_n}$ the collection of $u_{h_n} \in \mathcal{M}_{h_n}$ such that $\psi_{h_n} P_m^{h_n} u_{h_n} = \psi_0 P_m^0 u_0$ for some $u_0 \in D_0$.

As in the proof of Theorem \ref{GH convergence}, we can choose $N_1 \in \mathbb N$ such that the map $\hat{j}_{h_n} : \mathcal M_0 \to \mathcal M_{h_n}$ is an $\varepsilon/2$-isometry on $D_0$ and $\hat{i}_{h_n}: \mathcal{M}_{h_n} \rightarrow \mathcal{M}_0$ is an $\varepsilon/2$-isometry on $D_{h_n}$ for any $n \ge N_1$.
For given $T>0$ and $u_0 \in B_0$, let $u_0(t) = S(u_0,t)$ and $u_n(t) = S_{h_n}(\hat j_{h_n}(u_0),t)$ for $t \in [-T,T]$, and denote $\tilde{B}_0 = \{P_m^0 u_0(t) : u_0 \in B_0, t \in [-T,T]\}$.
Then we have
\begin{align*}
&\left\|\hat{j}_{h_n} (S_0 (u_0 (0),t)) - S_{h_n} (\hat{j}_{h_n} (u_0 (0)), t) \right\|_{L^2(\Omega_{h_n})} \\
&\leq \left\|\hat{j}_{h_n} (u_0 (t)) - j_{h_n} (u_0(t)) \right\|_{L^2(\Omega_{h_n})}
+ \left\|j_{h_n} (u_0(t)) - u_{n}(t) \right\|_{L^2(\Omega_{h_n})} \\
&\leq \left\|\hat{j}_{h_n} (u_0 (t)) - j_{h_n} (u_0(t)) \right\|_{L^2(\Omega_{h_n})}
+ \left\|j_{h_n} (p_0(t)) - p_{n}(t) \right\|_{L^2(\Omega_{h_n})}  \\
& \quad + \left\|j_{h_n} \Phi_0 (\psi_0 p_0(t)) - \Phi_{h_n} (\psi_{h_n} p_{n}(t)) \right\|_{L^2(\Omega_{h_n})}\\
&:= I_n + II_n+ III_n.
\end{align*}
By Lemma \ref{compare j}, we choose $N_2 > N_1$ such that 
$I_n < \varepsilon/6$ for any $n \geq N_2$.
By Lemma \ref{compare p(t)}, we take $N_3 > N_2$ such that
$$
II_n = \|j_{h_n} ( p_0 (t)) - p_{n} (t) \|_{L^2(\Omega_{h_n})} < \varepsilon/6, ~\forall n \ge N_3.
$$
Moreover, we have
\begin{align*}
III_n&= \|j_{h_n} \Phi_0 (\psi_0 p_0(t)) - \Phi_{h_n} (\psi_{h_n} p_{n}(t))\|_{L^2(\Omega_{h_n})}  \\
&\leq \|j_{h_n} \Phi_0 (\psi_0 p_0(t)) - \Phi_{h_n} (\psi_0 p_0(t)) \|_{L^2(\Omega_{h_n})}
+ \|\Phi_{h_n} (\psi_0 p_0(t)) - \Phi_{h_n} (\psi_{h_n} p_{n}(t)) \|_{L^2(\Omega_{h_n})} \\
&\leq \beta_{h_n} ( \psi_0 \tilde{B}_0) +  \left( \alpha(h_n) \sum_{i=1}^m |a_i (t)| + \|j_{h_n} p_0 (t) - p_{n} (t) \|_{L^2(\Omega_{h_n})} \right) \\
&= \beta_{h_n} (\psi_0 \tilde{B}_0) +  \alpha(h_n) \sum_{i=1}^m |a_i (t)| + II_n .
\end{align*}
Since $\alpha(h_n)$ and $\beta_{h_n} (\psi_0 \tilde{B}_0)$ converge to $0$ as $n \rightarrow \infty$, by Lemma \ref{compare p(t)}, we get $N_4 > N_3$ such that 
$$
III_n  < \varepsilon/6, ~\forall n \geq N_4.
$$
Consequently we derive that
$$
\left\|\hat{j}_{h_n} (S_0 (u_0 (0),t)) - S_{h_n} (\hat{j}_{h_n} (u_0 (0)), t) \right\|_{L^2(\Omega_{h_n})} < \dfrac{\varepsilon}{2}, ~\forall n \geq N_4.
$$

On the other hand, since $\hat{j}_{h_n}$ is an $\varepsilon/2$-isometry on $D_0$, we have
\begin{align*}
&\left\|\hat{i}_{h_n} (S_{h_n} (u_n(0) ,t)) - S_0(\hat{i}_{h_n}(u_n(0)),t) \right\|_{L^2(\Omega_0)} \\
&\leq \left\| S_{h_n} (\hat{j}_{h_n}(\hat{i}_{h_n} (u_n(0))),t) - \hat{j}_{h_n} (S_0 (\hat{i}_{h_n}(u_n(0))),t) \right\|_{L^2(\Omega_{h_n})} + \dfrac{\varepsilon}{2} < \varepsilon,~ \forall n \geq N_4.
\end{align*}
This shows that $D_{GH}^T(S_{h_n}|_{B_{h_n}}, S_0|_{B_0})< \varepsilon$ for all $n \geq N_4$.
The contradiction completes the proof.
\begin{flushright}
$\qed$
\end{flushright}

\vskip0.5cm
\noindent{\bf Acknowledgments.} The first author is supported by Basic Science Research Program through the National Research Foundation of Korea(NRF) funded by the Ministry of Education (NRF-2019R1A6A3A01091340).

\end{document}